\documentclass[fleqn,11pt]{article}
\usepackage[a4paper,width=180mm,top=25mm,bottom=25mm]{geometry}
\usepackage{authblk}
\usepackage[colorlinks=true,urlcolor=blue,citecolor=red,linkcolor=blue]{hyperref}
\usepackage{caption}
\usepackage{tikz-cd}
\usetikzlibrary{decorations.pathreplacing}
\usepackage{amsmath,amsthm,amssymb}
\usepackage{physics}

\newcommand{\set}[2]{\left\{#1\;|\;#2\right\}}
\newcommand{\pair}[2]{\langle #1,#2\rangle}
\newcommand{\id}{\mathrm{id}}
\newcommand{\Z}{\mathbb{Z}}
\newcommand{\R}{\mathbb{R}}
\newcommand{\C}{\mathbb{C}}
\newcommand{\from}{\colon}
\newcommand{\g}{\mathfrak{g}}
\newcommand{\dtzero}{\frac{d}{dt}\Big|_{t=0}}
\DeclareMathOperator{\Aut}{Aut}
\DeclareMathOperator{\pr}{pr}
\DeclareMathOperator{\Fr}{Fr}
\DeclareMathOperator{\GL}{GL}
\DeclareMathOperator{\Ad}{Ad}
\DeclareMathOperator{\Tor}{Tor}
\DeclareMathOperator{\STor}{STor}
\DeclareMathOperator{\Prin}{Prin}
\DeclareMathOperator{\SPrin}{SPrin}
\DeclareMathOperator{\Hom}{Hom}
\DeclareMathOperator{\Sym}{Sym}

\numberwithin{equation}{section} 

\newtheorem{Theorem}{Theorem}[section]
\newtheorem{Corollary}[Theorem]{Corollary}
\newtheorem{Lemma}[Theorem]{Lemma}
\newtheorem{Proposition}[Theorem]{Proposition}
\theoremstyle{definition}
\newtheorem{Definition}[Theorem]{Definition}
\newtheorem{Example}[Theorem]{Example}

\begin{document}

\title{Frames of group-sets and their application in bundle theory}
\author[1,2]{Eric J. Pap\thanks{Email address: \href{mailto:e.j.pap@rug.nl}{e.j.pap@rug.nl}}}
\author[1]{Holger Waalkens}
\affil[1]{Bernoulli Institute, University of Groningen, Groningen, The Netherlands}
\affil[2]{Van Swinderen Institute, University of Groningen, Groningen, The Netherlands}
\date{July 22, 2021}

\maketitle

\abstract{
We study fiber bundles where the fibers are not a group $G$, but a free $G$-space with disjoint orbits. These bundles closely resemble principal bundles, hence we call them semi-principal bundles. The study of such bundles is facilitated by defining the notion of a basis of a $G$-set, in analogy with a basis of a vector space. The symmetry group of these bases is a wreath product. Similar to vector bundles, using the notion of a basis induces a frame bundle construction, which in this case results in a principal bundle with the wreath product as structure group. This construction can be formalized in the language of a functor, which retracts the semi-principal bundles to the principal bundles. In addition, semi-principal bundles support parallel transport just like principal bundles, and this carries over to the frame bundle.
}

\vspace{3mm}


\vspace{3mm}


\section{Introduction}
Fiber bundles are mathematical objects generalizing the idea of a product space. This versatile definition now has a prominent role in many theories in both mathematics and physics. The basic principle is that the total space locally looks like a product of a small piece of the base and some fixed space, called the fiber. However, imposing a particular structure on the fiber reveals the existence of various specifications of fiber bundles, each with its distinctive properties. A well-known example is the vector bundle, where the fiber is a vector space.

Another famous class of bundles is formed by the principal bundles. However, here one must be careful when specifying the fiber. Given a topological or Lie group $G$, a principal $G$-bundle has a fiber given by the underlying space of $G$, endowed with some additional structure. The important observation here is that this is not the group structure, hence the fiber should not be identified with the original group $G$. Instead, the fiber has the structure of a $G$-space; a space endowed with a continuous, or even smooth, $G$-action. The fiber of a principal bundle is thus a $G$-torsor, i.e.\ a free and transtive $G$-space. A significant difference between $G$ as a group and $G$ as a torsor is the existence of a canonical base point. This is of utmost importance in gauge theory; the freedom in fixing a gauge only exists in the absence of such a canonical choice.

Given that principal bundles are thus fiber bundles based on torsors, and that torsors are specific group-spaces, clearly there exist more general classes of bundles with similar properties. That is, principal bundles form a specific subclass of the group-space bundles. In this paper we explore a certain part of this theory. Indeed, we will find that allowing not just torsors but general group-spaces opens up many possibilities, which at some point will force us to restrict ourselves. That is, we limit ourselves to free group-spaces that are a disjoint union of torsors, spaces we call semi-torsors. It turns out that these can be studied in great analogy to vector spaces. In particular, they admit a definition of a basis similar to vector spaces.

Bundles of semi-torsors we will call semi-principal bundles, due to their great similarity with principal bundles. Our original motivation to study these semi-principal bundles comes from the field of adiabatic quantum mechanics. This field studies how the energies and eigenstates of a Hamiltonian operator change under variation of system parameters. When looking at the underlying geometric structure, one naturally obtains semi-torsors. Consider a finite-dimensional Hamiltonian operator $H$. The set of all eigenstates of $H$ is then a $\C^\times$-space under the standard scaling action, where each eigenray is a $\C^\times$-torsor. Thus, if $H$ has distinct eigenvalues, then its eigenvector space is a $\C^\times$-semi-torsor. The varying of $H$ and the corresponding change in eigenstate can then be described using a semi-principal bundle, which for non-Hermitian Hamiltonians does not separate into principal bundles. The details of this physical model we treat in \cite{Pap2021AMechanics}; here we will treat the formal mathematics that is needed for it.

The set-up of the paper is as follows. We will start by considering the difference between a bundle of groups and a bundle of group-spaces in more detail in Sec.~\ref{sec:bundle of gps vs gp-spaces}. We will also discuss how the group-spaces induce a larger category of bundles that extends the principal bundles. We then focus on a special intermediate class of bundles, namely the semi-principal bundles whose fibers are semi-torsors. These semi-torsors and semi-principal bundles will be our main object of study for the rest of the paper. To start this study, we consider an equivalent of vector space bases for group-sets in Sec.~\ref{sec:bases of group-sets}. We will also focus on the symmetry group of these bases, and the functorial nature of sending admissible group-sets to their frame spaces. We then study the lifts of this technique to bundle theory in Sec.~\ref{sec:semi-prin and frame bundle}. We will find that the frames of a group-space can be used to define a frame bundle, similar to the frame bundle of a vector space. We finish the exploration with showing that semi-principal bundles support parallel transport just as well as principal bundles. In Sec.~\ref{sec:discussion} we will summarize the main results.

\section{Group bundles versus group-space bundles}
\label{sec:bundle of gps vs gp-spaces}

By definition, the key property of any fiber bundle is its local triviality. It is this property that allows one to endow the fiber with additional structure, which then lifts to a structure on the entire bundle. Let us take the fiber in a category $\mathcal{C}$ whose objects are topological spaces with additional structure and whose morphisms are continuous maps preserving this additional structure. A fiber bundle with fiber in $\mathcal{C}$ consists of the data $(\pi,B,M,F)$, where $F$ is an object of $\mathcal{C}$ called the model fiber, $B$ and $M$ are topological spaces and $\pi \from B \to M$ is a continuous surjection. This tuple is subject to the defining property that every point $x\in M$ has a neighborhood $U$ and a homeomorphism $\phi \from U\times F \to \pi^{-1}(U)$ such that
\begin{itemize}
    \item $\pi \circ \phi=\pr_U$,
    \item $\phi|_{\{u\}\times F}$ is an isomorphism in $\mathcal{C}$, for all $u\in U$.
\end{itemize}
Hence $B$, or more precisely each fiber $B_x$, should be equipped with the structure as given by $\mathcal{C}$, in such a way that a continuous structure on $B$ is obtained. In the following, we will take $\mathcal{C}$ to the category of topological groups, the category of group-spaces, as well as subcategories of these. The case of manifolds is then a natural adaptation.
Morphisms of fiber bundles follow the same general structure. A morphism from $(\pi_1,B_1,M_1,F_1) \to (\pi_2,B_2,M_2,F_2)$ is a pair $f=(f_1,f_2)$ of continuous maps $f_1 \from B_1\to B_2$ and $f_2\from M_1 \to M_2$ such that $\pi_2 \circ f_1=f_2\circ \pi_1$ and, for all $x\in M_1$, $f_1|_{(B_1)_x} \from (B_1)_x \to (B_2)_{f_2(x)}$, viewed in local trivializations, is a map in $\mathcal{C}$. In the following we will only consider the case where $f_2$ is identity, so we only specify $f_1$ and simply write $f \from B_1 \to B_2$.

\subsection{Group bundles}

We start with taking $\mathcal{C}$ to be the category of topological groups, which will result in the subclass of group bundles. We will first state the definition of a group bundle that we will use in this paper.
\begin{Definition}
    Let $G$ be a topological group. A \emph{group bundle with model fiber $G$} is a fiber bundle $(\pi,B,M,G)$ such that
    \begin{itemize}
        \item each fiber of $B$ has a group law. More precisely, writing $B\times_M B$ for the fiber-wise product bundle $\set{(b,b')\in B\times B}{\pi(b)=\pi(b')}$ over $M$, there is a bundle map
         \begin{equation}
            \begin{split}
               \mu\from B\times_M B &\to B\\
                (b_1,b_2)&\mapsto b_1b_2
            \end{split}
        \end{equation}
        such that $(B_x,\mu|_{B_x\times B_x})$ is a group for all $x\in M$. Moreover, the corresponding inverse map $B\to B$ should be continuous.
    \item each point $x\in M$ has a neighborhood $U$ and a homeomorphism $\phi \from U\times G \to \pi^{-1}(U)$ such that $\pi \circ \phi=\pr_U$, and for each $u\in U$ the map $\phi|_{\{u\}\times G} \from \{u\}\times G \to B_u$ is an isomorphism of groups.
    \end{itemize}
    A map of group bundles is a bundle map $f\from B \to B'$ such that $f|_{B_x} \from B_x \to B'_x$ is a morphism of groups for all $x\in M$. The group bundles and these maps together form the category of group bundles.
\end{Definition}
This definition is in line with e.g.\ \cite[p.~330]{Hatcher2001AlgebraicTopology} and \cite[Def.~2.8]{Mackenzie1989ClassificationBundle}. Here, we explicitly require that the group multiplications of the fibers fit together in a continuous bundle map. In addition, we also demand local triviality of the projection. Observe that vector bundles are a special case of group bundles; a vector bundle is a group bundle with a commutative model group and compatible scalar multiplication. It follows that all non-trivial vector bundles provide non-trivial group bundles, proving that not all group bundles are trivial.

We now wish to illustrate general group bundles by going over some basic results and examples. First, observe that transition functions take values in the group $\Aut(G)$. This clearly distinguishes group bundles from principal bundles, as the latter's transition functions take values in $G$, which is in general not isomorphic to $\Aut(G)$.
\begin{Lemma} \label{lem:trans func group bundle}
    The transition functions of a group bundle with model $G$ take values in the automorphism group $\Aut(G)$ of $G$.
\end{Lemma}
\begin{proof}
    Let us consider local trivializations $\phi_i\from U_i \times G \to B|_{U_i}$, $i=1,2$ such that the overlap $U:=U_1\cap U_2$ is non-empty. For $u\in U$, the maps $\phi_i|_{\{u\}\times G} \from G \to B_u$ are isomorphisms of groups. Hence the transition map $\phi_{21}=\phi_2^{-1} \phi_1 \from U\times G \to U\times G$ is identity on the $U$ component, hence restricting to any $u\in U$ yields an automorphism of $G$.
\end{proof}

Another property of group bundles is the existence of a global section. Given the fiber-wise group law axiom, each fiber $B_m$, with $m\in M$, is a group, hence in particular it has a unit $e_m:=e_{B_m}$. This gives rise to the map $\mathbf{1}\from M\to B$, $m\mapsto e_m$, which is a section of $\pi$, hence we call it the unit section of the group bundle. In the case of vector bundles, this section is known as the zero section, and provides an embedding of the base space $M$ in $B$. For general group bundles, there is the following statement.
\begin{Proposition} \label{prop:unit section}
    Given a group bundle $(\pi,B,M,G)$, the unit section $\mathbf{1} \from M\to B$ is an embedding. Moreover, if $G$ is discrete, then $B$ can be topologically partitioned into the image of $\mathbf{1}$, which is homeomorphic to $M$, and the complement of this image.
\end{Proposition}
\begin{proof}
    Clearly $\mathbf{1}$ is injective. Locally on $U\subset M$, $\mathbf{1}$ looks like the embedding $U \to U\times G$, $u\mapsto (u,e)$, which shows that $\mathbf{1}$ is an embedding. If $G$ is discrete, then $U\times \{e\}$ is closed and open in $U\times G$. Hence the image of $\mathbf{1}$ is locally disjoint from other parts of $B$, hence is so globally as $\mathbf{1}$ is global.
\end{proof}
Again we emphasize a significant difference between group bundles and principal bundles. It is well-known that a principal bundle admitting a global section must be trivial. Here we see that group bundles always admit a global section, but this does not imply triviality of the bundle. Indeed, any non-trivial vector bundle provides a counterexample, but we will now go over some other examples. We will take $M$ to be the circle $S^1$. In this case, we may view a group bundle $B$ with model group $G$ as a quotient of the trivial group bundle $[0,1]\times G$ by gluing $\{0\}\times G$ and $\{1\}\times G$ by an automorphism of $G$.
\begin{Example}[$\Z_3$-bundles over the circle] \label{ex:z3 bundles circle}
    Let us view $\Z_3$ as the additive group of remainders modulo 3, whose elements are $\bar{1}$, $\bar{0}$ and $-\bar{1}$.
    The group $\Aut(\Z_3)$ consists of 2 elements, namely $\id$ and $-\id$. Naturally, gluing via $\id$ will yield the trivial bundle over $S^1$, i.e.\ the bundle $B$ is isomorphic to the group bundle $S^1\times \Z_3$. The unit section traces out the subspace $S^1\times \{\bar{0}\}$, which is indeed well separated from its complement and homeomorphic to the base $S^1$.
    
    Gluing along $-\id$ means that $\bar{1}$ and $-\bar{1}$ are exchanged upon returning, hence return to themselves after a second round. We may thus picture them to traverse the boundary of a M{\"o}bius band. The element $\bar{0}$, which follows the unit section, will again trace out a circle. Hence the total space of this bundle consists of a circle and the boundary of a M{\"o}bius band, hence is homeomorphic to 2 disjoint circles. It thus cannot be homeomorphic to the trivial bundle, as this consists of 3 disjoint circles. Compare the illustrations of the bundles in Fig.~\ref{fig:z3 bundles over circle}.
    
    \begin{figure}[h]
        \centering
        \includegraphics[width=0.4\textwidth]{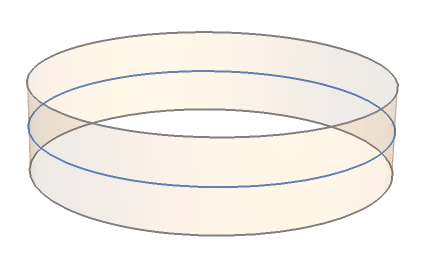}
        \includegraphics[width=0.4\textwidth]{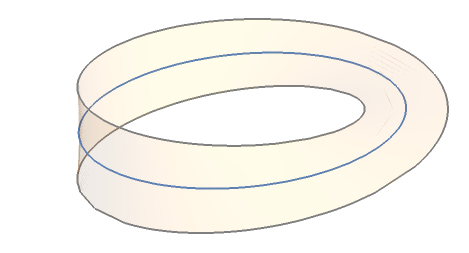}
        \caption{Representatives of the two non-isomorphic classes of $\Z_3$-bundles over the circle, see Ex.~\ref{ex:z3 bundles circle}. The shaded cylinder and M{\"o}bius band are visual aids. In both cases, the unit section traces out the middle circle, whereas the other group elements follow the boundary.}
        \label{fig:z3 bundles over circle}
    \end{figure}
\end{Example}
Observe that this example treats the most elementary non-trivial group bundle. Indeed, the circle requires at least 1 automorphism to be specified, and the smallest group possessing non-trivial automorphisms is $\Z_3$. Naturally, if $G$ is trivial, then $B\cong M$ is trivial as the fiber has only 1 point. In case $G=\Z_2$ the bundle will also be trivial; one can use Lemma~\ref{lem:trans func group bundle} to conclude that all transition functions are trivial, or alternatively use Prop.~\ref{prop:unit section} to remark that $B$ consists of two disjoint copies of $M$, hence is trivial.

The argument in Ex.~\ref{ex:z3 bundles circle} holds for any group $G$ with $\Aut(G) \cong \Z_2$. Indeed, the identity map in $\Aut(G)$ yields the trivial bundle $S^1\times G$, while the non-identity map yields a non-trivial bundle. For instance, picking $G=\Z$ similarly yields infinitely many copies of the boundaries in Fig.~\ref{fig:z3 bundles over circle}, which can be envisioned to lie on the infinite cylinder resp.\ infinite M{\"o}bius band. The choice $G=U(1)$ reproduces the familiar picture of a cylinder whose edge circles are glued to obtain a torus resp.\ a Klein bottle. As $\Z_4$ is naturally a subgroup of $U(1)$, the case $G=\Z_4$ can again be viewed as lines on a surface, where the surface is now given by a $U(1)$-bundle. Pictorially, one places 4 equidistant points on a circle, which become straight lines on a cylinder, and by gluing the edge circles of this cylinder one finds all lines coming back to themselves on the torus resp.\ 2 lines coming back and 2 lines exchanging on the Klein bottle.

A final remark is that the (isomorphism class of) a group bundle over the circle may be the same for different automorphisms of $G$. For one, if the automorphisms can be connected by a path, then the resulting bundles are isomorphic. As an example, if $G=\R$, then $\Aut(\R)=\R^\times\cong \R_{>0}\times \Z_2$, and the class of the obtained bundle only depends on the $\Z_2$ part. In addition, conjugate automorphisms of $G$ also yield isomorphic bundles. This we illustrate in Ex.~\ref{ex:z_2xz_2 bundles over circle}.
 
\begin{Example}[$\Z_2\times\Z_2$-bundles over the circle] \label{ex:z_2xz_2 bundles over circle}
    Let us consider $G=\Z_2\times \Z_2\cong \{1,a,b,c\}$, where $a^2=1$, $ab=ba=c$ and similarly for the cyclic permutations of these equations. The elements $a$, $b$ and $c$ thus satisfy identical relations, and indeed the automorphisms are given by permuting $a$, $b$ and $c$ amongst themselves. Hence $\Aut(G)$ is isomorphic to $S_3$.
    
    Although this group has 6 elements, there are less isomorphism classes. Let us go over all permutations in $\Aut(G)$. Naturally, the identity permutation will yield the trivial group bundle $S^1\times G$. All the transpositions will yield another class, as do the 3-cycles.
    
    Take $(ab)$ for instance. This means that we glue the 4 intervals in $[0,1]\times G$ such that the intervals of $a$ and $b$ form a single circle, whereas the intervals of 1 and $c$ will each form a separate circle. For $(bc)$ and $(ca)$ one will get essentially the same picture, revealing that the obtained bundles are isomorphic.
    For $(abc)$ and $(acb)$ a similar argument holds; now the intervals of $a$, $b$ and $c$ are connected to a single circle, and the only difference between the two gluings is the order in which the pieces of the circle are traversed.
    
    We thus see that the obtained bundle class is the same for automorphisms in the same conjugacy class of $\Aut(G)$. Moreover, each conjugacy class yields another class of group bundles; the obtained spaces are homeomorphic to 4, 3 and 2 circles respectively.
\end{Example}

\subsection{Group-space bundles}

Let us shift our attention to group-sets instead of groups. This brings us closer to principal bundles as these are a special class of group-space bundles. Indeed, the axioms on the fiber of a principal bundle are similar to the axioms of a group-set. We will investigate these objects in more detail here and later on in the paper, and so wish to review some basic statements on group-sets.
\begin{Definition}
    Let $G$ be a group with unit $e$. A $G$-set is a tuple $(F,A)$ where $F$ is a set and $A$ is a left action of $G$ on $F$, i.e.\ a map
    \begin{equation}
        A\from G \times F \to F, \quad    (g,f)\mapsto gf
    \end{equation}
    such that $\forall f\in F$ and $\forall g_1,g_2\in G$ one has $g_1(g_2f)=(g_1g_2)f$ and $ef=f$. For convenience, we will simply refer to the $G$-set as $F$, which we understand as a set endowed with a left action $A$ by $G$. In case $G$ is a topological group, $F$ a topological space and the action map $A$ is continuous, $F$ is called a $G$-space. Similarly, if $G$ is a Lie group, $F$ is a manifold and the action map $A$ is smooth, then $F$ is a $G$-manifold.
    
    A function $\alpha \from F\to F'$ between a $G$-set $F$ and a $G'$-set $F'$ is called $\xi$-equivariant, where $\xi \from G \to G'$ is a homomorphism, if
    \begin{equation}
        \alpha(gf)=\xi(g)\alpha(f), \quad \forall f\in F, \forall g\in G.
    \end{equation}
    In particular, a $G$-map is a function between two $G$-sets that is $\id_G$-equivariant, i.e.\ $\alpha(gf)=g\alpha(f)$.
    In case $F$ and $F'$ are $G$-spaces or $G$-manifolds, the map $\alpha$ is also required to be continuous resp.\ smooth. Moreover, the category formed by $G$-sets and $G$-maps is called the category of $G$-sets. Similarly, one has the category of $G$-spaces and $G$-manifolds.
    
    A $G$-set is called free resp.\ transitive if the action is free resp.\ transitive, which is equivalent to the map
    \begin{equation}
        \begin{split}
            G\times F &\to F \times F\\
            (g,f) &\mapsto (gf,f)
        \end{split}
    \end{equation}
    being injective resp.\ surjective. Denoting the action quotient map as $q\from F \to F/G$, the image of the above map, i.e.\ all pairs of elements in the same orbit, can be written as
    \begin{equation}
        F\times_q F=\set{(f',f)\in F\times F}{q(f)=q(f')}.
    \end{equation}
    If $F$ is both free and transitive, then $F$ is called a principal homogeneous space for $G$, or $G$-torsor.
\end{Definition}
The fibers of a principal $G$-bundle are thus $G$-torsors. Of course, a particular $G$-torsor is the underlying space of $G$ where the action is given by left translation. Clearly, any $G$-torsor $F$ is isomorphic to this $G$-torsor; any orbit map establishes an isomorphism of $G$-torsors. A straightforward generalization of principal bundles are thus group-space bundles, which can be defined as follows.
\begin{Definition}
    Let $G$ be a topological group and $F$ a $G$-space. A \emph{$G$-space bundle with model $F$} is a fiber bundle $(\pi,B,M,F)$ such that
    \begin{itemize}
        \item each fiber of $B$ is a $G$-space. That is, there is a bundle map
         \begin{equation}
            \begin{split}
               A \from G\times B &\to B\\
                (g,b)&\mapsto gb
            \end{split}
        \end{equation}
        such that $(B_x,A|_{G\times B_x})$ is a $G$-space for all $x\in M$.
        \item each point $x\in M$ has a neighborhood $U$ and a homeomorphism $\phi \from U\times F \to \pi^{-1}(U)$ such that $\pi \circ \phi=\pr_U$, and for each $u\in U$ the map $\phi|_{\{u\}\times F} \from \{u\}\times F \to B_u$ is an isomorphism of $G$-spaces.
    \end{itemize}
    Let $B$ be a $G$-space bundle and $B'$ be a $G'$-space bundle. A map of group-space bundles $f\from B \to B'$ is a bundle map such that $f|_{B_x}\from B_x \to B'_x$ is $\xi$-equivariant for all $x\in M$, i.e.\ $f(gb)=\xi(g)f(b)$ for all $g\in G$ and $b\in B$, where $\xi \from G \to G'$ is a homomorphism.
\end{Definition}
This defines the category of group-space bundles. From it, one obtains special subcategories by restricting the model group-space to a prescribed class. The principal bundles are readily recognized as the restriction of group-space bundles to the case where the model fiber is a torsor. The group-space bundles are more general in the following sense.
\begin{Lemma}
    The principal bundles are a full subcategory of the group-space bundles.
\end{Lemma}
\begin{proof}
    As a $G$-torsor is a $G$-space, any principal bundle is a group-space bundle. In addition, any map of principal bundles is a map of group-space bundles. This proves the subcategory claim. It is full as any group-space bundle map between principal bundles is an equivariant bundle map, hence a map of principal bundles.
\end{proof}

A significant difference between principal bundles and general group-space bundles is the way in which the structure group can describe changes in the fiber. This is relevant already for the transition functions, but also when one considers holonomy. The nature of the transition functions can be found in the standard way, as done for group bundles previously, and results in the following.
\begin{Lemma}
    The transition functions of a group-space bundle with model fiber $F$ take values in the automorphism group $\Aut(F)$, i.e.\ the group of invertible group-space maps $F\to F$.
\end{Lemma}
Let us first relate this back to principal $G$-bundles, in which case $F$ may be taken as the $G$-torsor $G$. The left-equivariant invertible maps are exactly the right-translations, hence $\Aut(F)\cong G$. In other words, the group $G$ faithfully describes all relevant changes in the fiber $F$. This need not be so for general group-spaces $F$. Intuitively speaking, if $F$ is not free then $G$ has redundant elements, and if $F$ is not transitive then $G$ is too small. A more formal description of $\Aut(F)$ is given in Appendix~\ref{sec:Aut of G-set}.

Another significant difference is the internal structure of the fiber $F$. By this we mean that $F$ itself can have an interesting topology. Let us consider the quotient map $q\from F\to F/G$. Remarkably, if $q$ is locally trivial, the tuple $(q,F,F/G,G)$ is itself a group-space bundle. This bundle need not be trivial, in fact one can take $F$ to be (the total space of) any non-trivial principal $G$-bundle. This means that non-trivial topology can appear not only on the global level, but also already on the fiber level.

We will restrict ourselves to a certain class of group-spaces, for which the fibers have both a clear automorphism group and the quotient map is trivial. This can be done by requiring the group-spaces to be free, but instead of having a single orbit they are allowed to have more, as long as these form a partition of the group-space in open and closed subsets. In other words, we consider $G$-spaces $F$ such that $F$ is free and $F/G$ is discrete. It readily follows that $F$ must be isomorphic to $G\times F/G$ as $G$-spaces, which reveals the trivial topology. As all such $G$-spaces are of the form $\sqcup_{F/G} G$, i.e.\ a disjoint union of `standard' torsors, we propose to call them semi-torsors. We will start phrasing the results for manifolds, and leave the topological equivalents as implicit.
\begin{Definition}
    Let $G$ be a Lie group and $F$ a $G$-manifold. Then $F$ is a \emph{semi-torsor} if $F$ is free and $F/G$ is discrete.
\end{Definition}
Clearly, every torsor is a semi-torsor; in this case $F/G$ consists of a single point, hence is discrete. As principal bundles correspond to torsors, the bundles corresponding to the semi-torsors will form a generalization of the principal bundles. We call these bundles semi-principal bundles as they locally look like a sum of principal bundles. A schematic overview of how these bundles relate to group-spaces and other classes of bundles can be found in Fig.~\ref{fig:overview group-space and group-space bundle}.
\begin{Definition}
    Given a Lie group $G$, a \emph{semi-principal $G$-bundle} is a group-space bundle $\pi \from B \to M$ where the model fiber $F$ is a $G$-semi-torsor. That is, $B$ is endowed with a fiber-preserving $G$-action, local trivializations of $B$ are $G$-equivariant, and the model fiber $F$ is a free $G$-manifold such that $F/G$ is discrete.
\end{Definition}

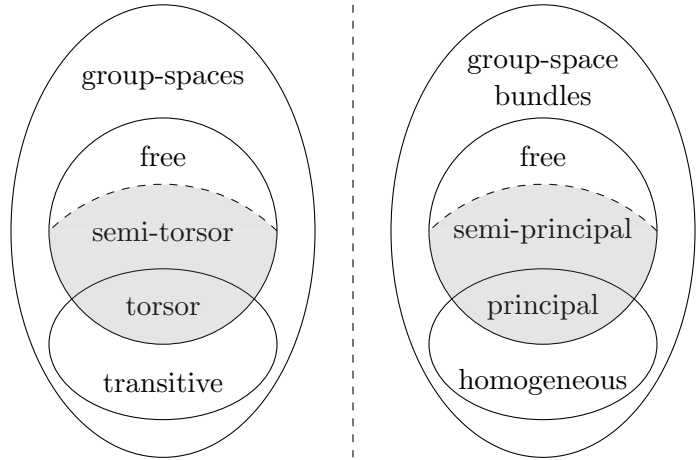
\begin{figure}[h]
    \centering
    \begin{tikzpicture}
        \draw (0,2) ellipse (2 and 3) (0,4) node{group-spaces};
        \draw (0,2) circle[radius=1.5] (0,3) node{free};
        \draw (0,0.5) ellipse (1.5 and 1) (0,0) node{transitive};
        \draw[dashed] (1.5,2) arc (45:135:2.121) (0,1) node{torsor} (0,2) node{semi-torsor};
        \fill[color=gray, opacity=0.2] (1.5,2) arc (45:135:2.121) -- (1.5,2) arc (0:-180:1.5);

        \draw (5,2) ellipse (2 and 3) (5,4) node[align=center]{group-space\\ bundles};
        \draw (5,2) circle[radius=1.5] (5,3) node{free};
        \draw (5,0.5) ellipse (1.5 and 1) (5,0) node{homogeneous};
        \draw[dashed] (6.5,2) arc (45:135:2.121) (5,1) node{principal} (5,2) node{semi-principal};
        \fill[color=gray, opacity=0.2] (6.5,2) arc (45:135:2.121) -- (6.5,2) arc (0:-180:1.5);
        
        \draw[dashed] (2.5,5) -- (2.5,-1);
    \end{tikzpicture}
    \caption{Correspondence between classes of group-spaces and the classes of group-space bundles having these as fibers. This paper focuses on the gray parts, i.e.\ the semi-torsors and semi-principal bundles, which extend the torsors and principal bundles.}
    \label{fig:overview group-space and group-space bundle}
\end{figure}

Unlike principal bundles, the projection map $\pi \from B \to M$ of a semi-principal bundle need not coincide with the action quotient $Q \from B \to B/G$, where $B/G$ is endowed with the unique smooth manifold structure such that $Q$ becomes a surjective submersion. Indeed, when taking the action quotient, the fiber $F$ is reduced to $F/G$, which need not consist of a single point. Nevertheless, this new fiber is a discrete space, and as proven below $B/G$ forms a covering of $M$. As $Q$ will still define a principal bundle, the following result says that any semi-principal $G$-bundle is a principal $G$-bundle on top of a covering space.
\begin{Proposition} \label{prop:semi-prin = prin + covering}
    Let $\pi \from B\to M$ be a semi-principal $G$-bundle with model fiber $F$, set $X=F/G$. The map $Q \from B \to B/G$ defines a principal $G$-bundle. Moreover, the quotient space $B/G$ is itself an $X$-bundle over $M$, i.e.\ a regular $|X|$-fold covering. This bundle is defined by the reduced map $\pi/G \from B/G \to M$ given by the relation $\pi =\pi/G \circ Q$, i.e.\
    \begin{equation}
        \begin{tikzcd}[column sep=1em]
            B \ar[swap]{rd}{\pi} \ar{rr}{Q} && B/G \ar[dashed]{ld}{\pi/G}\\
            & M
        \end{tikzcd}
    \end{equation}
    is a commutative diagram of semi-principal bundles.
\end{Proposition}
\begin{proof}
    As $\pi$ is $G$-invariant, it is constant on the fibers of $Q$. As $Q$ is a smooth submersion, the map $\pi/G$ is a well-defined smooth map \cite[Thm 4.30]{Lee2012IntroductionManifolds}. To see the local forms of $Q$ and $\pi/G$, let us use a $G$-equivariant local trivialization $\phi \from U \times X\times G \to B|_U$ of $\pi$. One can reduce $\phi$ to orbits, and as $(B|_U)/G=(B/G)|_U$ this results in the commutative diagram
    \begin{equation}
        \begin{tikzcd}[column sep=1em]
            U \times X \times G \ar{d}{\phi} \ar{rr}{\pr_{U\times X}} && U \times X \ar[dashed]{d}{\phi/G}\\
            B|_U \ar{rr}{Q} \ar[swap]{rd}{\pi} && (B/G)|_U \ar{ld}{\pi/G}\\
            & U
        \end{tikzcd}
    \end{equation}
    where all solid arrows are smooth.
    
    Let us consider the induced map $\phi/G$. It is smooth as $Q\circ \phi$ is a smooth map that is constant on the fibers of the smooth submersion $\pr_{U\times X}$. In fact, $\phi/G$ is a diffeomorphism. Being the quotient of a bijection, it is bijective. Its inverse is smooth as well; $\pr_{U\times X}\circ \phi^{-1}$ is a smooth map that is constant on the fibers of the smooth submersion $Q$. Hence $\phi/G$ is a local trivialization of $\pi/G$ provided that $\pi/G \circ \phi/G=\pr_{U}$. The latter can be deduced from
    \begin{equation}
        (\pi/G \circ \phi/G)(u,x)=(\pi/G \circ Q\circ \phi)(u,x,g)=(\pi\circ \phi)(u,x,g)=\pr_U(u,x,g)=u,
    \end{equation}
    where $g\in G$ may be chosen arbitrarily because of $G$-equivariance. Hence $\pi/G$ is an $X$-bundle over $M$, and as $X$ is discrete this is a covering.
    
    Shifting our attention to $Q \from B\to B/G$, using the same maps one can find a local trivialization for $Q$. This is $\phi\circ ((\phi/G^{-1}\times \id_G) \from (B/G)|_U \times G \to B|_U$, $([b],g) \mapsto \phi((\phi/G)^{-1}([b]),g)$. This diffeomorphism is $G$-equivariant as $\phi$ is, and satisfies 
    \begin{equation}
        Q \circ \phi\circ (\phi/G^{-1}\times \id_G)=\phi/G\circ \pr_{U\times X}\circ (\phi/G^{-1}\times \id_G)=\phi/G\circ \phi/G^{-1}\circ \pr_{(B/G)|_U}=\pr_{(B/G)|_U},
    \end{equation}
    hence is a local trivialization of $Q$. This implies $Q$ is a principal $G$-bundle.
\end{proof}

This result clearly reveals two special cases of semi-principal bundles. Naturally, one case is where the bundle is actually principal, which means that $\pi$ and $Q$ coincide. The other special case is that $\pi$ and $\pi/G$ coincide, which happens if and only if $G$ is trivial. Here the definition of semi-principal bundle reduces to that of a fiber bundle with discrete fiber, i.e.\ a regular covering space.

An example of a generic semi-principal bundle can be found in Ex.~\ref{ex:winding torus 1} below. Here we consider the trivial principal $U(1)$-bundle over the circle $S^1$, i.e.\ the torus, but we wind it around multiple times to obtain a coil-like object. The resulting projection is no longer a principal bundle, but it does remain semi-principal. We will repeatedly build upon this example throughout the paper by illustrating new results with it.

\begin{Example}[Winding torus] \label{ex:winding torus 1}
    Let us consider the torus $T^2=S^1\times S^1$, endowed with the standard $U(1)$-action on its second factor. The projection $\pi \from T^2 \to S^1$ on the first factor then realizes the torus as a trivial principal $U(1)$-bundle over the circle. However, we can also view $T^2$ as a semi-principal bundle over $S^1$ in the following way. Consider the map $s_k \from S^1 \to S^1$, with $k\geq 1$ an integer, given by winding the circle $k$ times around, i.e.\ $z\mapsto z^k$ for $z$ on the unit circle in $\C$. Then $\Pi_k:=s_k\circ \pi \from T^2 \to S^1$ winds the torus $k$ times around the circle. The picture to have in mind is illustrated in Fig.~\ref{fig:winding torus} for the case $k=2$. Clearly, $\Pi_k$ defines a semi-principal $U(1)$-bundle with model fiber $\sqcup^k S^1$, i.e.\ $k$ disjoint circles. For $k>1$, so that $\Pi_k\ne \pi$, the bundle $\Pi_k$ is non-trivial, as can be deduced using topology. The trivial semi-principal $U(1)$-bundle over $S^1$ with model fiber $\sqcup^k S^1$ is given by the projection $S^1\times (\sqcup^k S^1) \to S^1$ on the first factor, in particular the total space consists of $k$ disjoint tori. As the total space of $\Pi_k$ is always a single torus, it follows that $\Pi_k$ is non-trivial for $k>1$.
    
    \begin{figure}[h]
        \centering
        \includegraphics[width=0.5\textwidth]{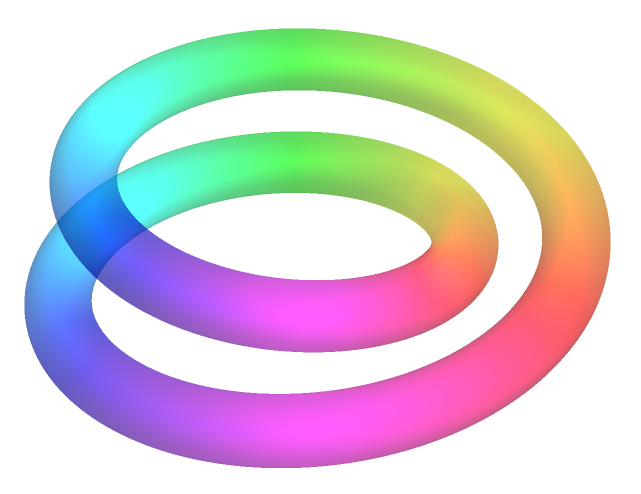}
        \caption{Illustration of the total space of $\Pi_2$ in Ex.~\ref{ex:winding torus 1}. This space is topologically a torus, but the projection $\Pi_2$ indicates that it winds twice around its base circle. Points with the same color are projected to the same point in $S^1$.}
        \label{fig:winding torus}
    \end{figure}
\end{Example}

One observes that the gluing needed to obtain $\Pi_k$ cannot be described using an element in $U(1)$ if $k>1$. Here lies our main motivation for the following material; one can pass over to a principal bundle formulation such that the gluing, and also holonomies, can again be described using the structure group.

\section{Bases of a group-set} \label{sec:bases of group-sets}

In this section we will study a notion of basis, also known as frame, of a group-set. In particular, we will consider the space formed by all the frames of a group-set. This will include the study of the symmetry group of these frames, which will be a wreath product. This group will reappear in Sec.~\ref{sec:semi-prin and frame bundle} as the structure group when we return to bundle theory.

\subsection{Basis of a group-set}

Similar to the notion of basis of a vector space, one can define the basis of a $G$-set $F$. Intuitively speaking, such a basis should be a tuple $(f_x)_{x\in X}$ of elements of $F$, where $X$ is the indexing set, such that any point in $F$ can be `reached' using these points in a unique way, and that a $G$-map from $F$ to any other $G$-set is specified once the images of the $f_x$ are known. We will use that a tuple $(f_x)_{x\in X}$ is equivalent to a function $\tilde{f} \from X \to F$ by the relation $\tilde{f}(x)=f_x$, and we will often use the $\tilde{f}$ formalism throughout this paper for convenience. We will use tilde notation to indicate tuples throughout this paper.

Let us first discuss how this `reaching of points' can be phrased in a formal way, staying in close analogy to vector spaces. In linear algebra, if $V$ is a real vector space, then given a tuple $\tilde{v}:=(v_1,\dots,v_m)$ of vectors in $V$, there is the associated map from coefficients to the vector space as
\begin{equation}
    \begin{split}
        \phi_{\tilde{v}} \from \R^m &\to V\\
        (c_1,\dots,c_m) &\mapsto \sum_{i=1}^m c_iv_i.
    \end{split}
\end{equation}
As this map is already linear, it is an isomorphism of vector spaces if and only if it is bijective. This question is usually subdivided in the properties of linear independence and spanning set. That is, the tuple $\tilde{v}$ is linearly independent if and only if $\phi_{\tilde{v}}$ is injective, and similarly the tuple is a spanning set if and only if $\phi_{\tilde{v}}$ is surjective. If $V$ admits a finite basis, then all bases of $V$ have the same number of vectors, and the dimension of $V$ is defined to be this finite number. If not, then $V$ is called infinite-dimensional.

This argument can almost fully be copied once we agree on a specific $G$-map induced by a tuple $(f_x)_{x\in X}$ of $G$-set elements. Of course, in this case no addition $+$ is available, but there is some sort of scaling given by the group action; $(g,f)\mapsto gf$. It is now straightforward to extend this map to tuples; we define this map below and also extend the definition of a basis to $G$-sets. Here we agree on the canonical $G$-set structure on $G \times X$ given by the action $g(h,x)=(gh,x)$.
\begin{Definition}
    Let $X$ be an index set, and consider the set $F^X$ of functions $\tilde{f}\from X \to F$, or equivalently the set of ordered tuples $(f_x)_{x\in X}$ with $f_x\in F$ for all $x\in X$. Given such a tuple, we define its associated map to be the $G$-map
    \begin{equation}
        \begin{split}
            \phi_{\tilde{f}} \from G \times X &\to F\\
            (g,x) &\mapsto g\cdot f_x=g\cdot \tilde{f}(x).
        \end{split}
    \end{equation}
    A tuple $\tilde{f}$ is called a \emph{basis} of $F$ if the associated map $\phi_{\tilde{f}}$ is an isomorphism of $G$-sets. The subset of $F^X$ consisting of all bases of $F$ we denote as $\Fr(F)$.
\end{Definition}

Before discussing the existence of these bases, let us first consider the following. Namely, the scaling of vector spaces may differ in a crucial way from the group action used here. If one wants to preserve some idea of linear independence, it is important that $g_1f=g_2f$ implies $g_1=g_2$. However, this need not hold for a general group action; it holds only for free actions. Hence, one may expect that the idea of a basis only matches intuition when $F$ is free. This is indeed sufficient; the map $\phi_{\tilde{f}}$ is an isomorphism if and only if it is bijective, so we need both injectivity and surjectivity. The map $\phi_{\tilde{f}}$ is surjective if and only if every orbit contains at least one of the $f_x$, and if $F$ is free, then $\phi_{\tilde{f}}$ is injective if and only if the $f_x$ lie in different orbits. This indeed mimics the properties of linear independence and spanning set. In fact, having a basis and being free is equivalent, as stated in the following lemma. We also prove the intuitive statement that all bases of a given $G$-set have the same cardinality; this `dimension' is now the number of orbits $|F/G|$, which can be infinite.
\begin{Lemma} \label{lem:G-basis iff free}
    A $G$-set $F$ has a basis if and only if $F$ is free. Moreover, for a free $G$-set $F$ the index set $X$ is unique up to bijection, in particular $X\cong F/G$.
\end{Lemma}
\begin{proof}
    If $F$ has a basis, then $F\cong G\times X$ as $G$-sets. As the latter is free, so is $F$. Conversely, assume $F$ is free. Then every orbit in $F$ is bijective to $G$. We only need to specify base points for each orbit; these base points will form the basis. That is, for each orbit $x \in F/G$, pick a point in this orbit and denote it by $s(x)$. This defines a section $s$ of the quotient map $q \from F \to F/G$. To check that $s$ is a basis of $F$, we consider its induced map $\phi_s \from G \times F/G \to F$ given by $(g,x) \mapsto g s(x)$. This has an explicit inverse 
    \begin{equation}
        f\mapsto ([f/s(q(f))],q(f)),
    \end{equation}
    where $[f/s(q(f))] \in G$ is the group element such that $f=[f/s(q(f))]s(q(f))$. This element is well-defined; it exists as $f$ and $s(q(f))$ lie in the same orbit, and it is unique as $F$ is free. Hence $s$ is a basis of $F$ with index set $X=F/G$.
    
    Assuming $F$ is free, pick two bases $\tilde{f}_1$ and $\tilde{f}_2$ with index sets $X_1$ resp.\ $X_2$. Then $\phi_{\tilde{f}_2}^{-1}\phi_{\tilde{f}_1} \from G\times X_1 \to G \times X_2$ is a $G$-isomorphism. Hence the orbit sets are bijective, which yields $X_1\cong X_2$. Hence all index sets of the $G$-set $F$ are bijective. We just saw that $F/G$ is always an index set, hence any other index set $X$ must be bijective to $F/G$.
\end{proof}
Given this result, we will restrict ourselves to free $F$ from now on. Also, we take $X=F/G$, but sometimes still write $X$ for clarity. In addition, one may observe that in the previous proof one does need $s$ to be a section; an orbit may be sent to a point in another orbit, as long as all orbits are visited exactly once. We like to emphasize this freedom at certain places by not using the quotient notation.

We remark that $\phi_{\tilde{f}}$ has an interesting interpretation concerning the model of a $G$-set. Given a finite-dimensional real vector space $V$, one often identifies it with $\R^m$ for some $m$ using a linear isomorphism. In this way, the vector spaces $\R^m$, $m\in \mathbb{N}$, provide models for real (finite-dimensional) vector spaces. One may then ask what the model space is in case of $G$-sets, and the above indicates that this is $G\times X$. This is valid for free $G$-sets, as any free $F$ is isomorphic to $G\times F/G$, hence of the correct form. The bases are then the labels of such isomorphisms $F \cong G \times X$.

We finish this inspection of individual bases with the following statements on $G$-maps in relation to bases. Again, we emphasize the similarity with the theory of vector spaces.
\begin{Lemma} \label{lem:G-map properties}
    Let $F,F'$ be $G$-sets and $\alpha \from F \to F'$ a $G$-map. Then:
    \begin{enumerate}
        \item for any $\tilde{f}\in F^X$, one has $\phi_{\alpha(\tilde{f})}=\alpha \circ \phi_{\tilde{f}}$,
        \item if $F$ is free, the map $\alpha$ is fully determined by its values on a basis,
        \item if both $F$ and $F'$ are free, then $\alpha$ is an isomorphism if and only if $\alpha$ sends a basis to a basis.
    \end{enumerate}
\end{Lemma}
\begin{proof}
    For the first, observe $\phi_{\alpha(\tilde{f})}(g,x)=g \alpha(f_x)=\alpha(gf_x)=\alpha \circ \phi_{\tilde{f}}(g,x)$, which indeed yields $\phi_{\alpha(\tilde{f})}=\alpha \circ \phi_{\tilde{f}}$.
    For the second, assume $\beta \from F \to F'$ is another $G$-map that coincides with $\alpha$ on a basis $\tilde{f}$ of $F$. Pick any $f\in F$; it is related to the basis $\tilde{f}$ via $f=hf_x$ for some $h\in G$ and $x\in X$. Using this decomposition, one finds
    \begin{equation}
        \alpha(f)=\alpha(hf_x)=h\alpha(f_x)=h\beta(f_x)=\beta(hf_x)=\beta(f)
    \end{equation}
    and as $f$ was arbitrary, $\alpha=\beta$.
    The third then follows as well; given that $\phi_{\tilde{f}}$ is a $G$-isomorphism, then $\alpha$ is a $G$-isomorphism if and only if $\phi_{\alpha(\tilde{f})}$ is, i.e.\ if and only if $\alpha(\tilde{f})$ is a basis of $F'$.
\end{proof}

\subsection{The space of bases of a group-set}

We will now shift our scope and study properties of the bases of $F$ as a whole, i.e.\ we study $\Fr(F)$. This we do using a kind of basis criterion in terms of the quotient map $q$. Basically, we pose the question when a candidate basis $\tilde{f}$, i.e.\ any element in $F^X=\{\tilde{f}\from X\to F\}$ for a chosen index set $X$ bijective to $F/G$, lies in $\Fr(F)$.

The basic observation is that invertibility of $\phi_{\tilde{f}}$, as it is equivariant, can be deduced by just knowing how the elements $f_x$ are distributed over the orbits of $F$, i.e.\ by knowing the map $q\circ \tilde{f} \from X\to F/G$. More explicitly, $q\circ \tilde{f}$ fits in the commutative diagram
\begin{equation} \label{diag:phi_f and q f}
    \begin{tikzcd}
        G\times X \ar{r}{\phi_{\tilde{f}}} \ar{d}{\pr_X} & F \ar{d}{q}\\
        X \ar{r}{q\circ \tilde{f}} & F/G
    \end{tikzcd}
\end{equation}
which indicates that $q\circ \tilde{f}$ is $\phi_{\tilde{f}}$ reduced to orbits. This yields the following short yet convenient criterion for bases.
\begin{Lemma} \label{lem:G-basis iff generalized section}
    Let $F$ be a free $G$-set. Then $\tilde{f}\in F^X$ is a basis if and only if $q\circ \tilde{f}$ is bijective.
\end{Lemma}
\begin{proof}
    First, observe that $\phi_{\tilde{f}}$ is injective if and only if $\tilde{f}$ reaches every orbit at most once. In other words, every fiber of $q$ contains at most one element of the basis $\tilde{f}$. This is equivalent to $q\circ \tilde{f}$  being injective.
    Dually, the map $\phi_{\tilde{f}}$ is surjective if and only if every orbit is reached at least once. In other words, every fiber of $q$ contains at least one element of the basis $\tilde{f}$. This is equivalent to $q\circ \tilde{f}$ being surjective.
\end{proof}

Hence, the question whether a particular $\tilde{f}\in F^X$ lies in $\Fr(F)$ can be answered by inspecting whether $q\circ \tilde{f}$ is bijective. We are thus inclined to look at the map $\tilde{f} \mapsto q\circ \tilde{f}$, which we formally write as
\begin{equation}
    \begin{split}
        q^X \from F^X &\to X^X\\
        \tilde{f} &\mapsto q\circ \tilde{f}.
    \end{split}
\end{equation}
In other words, this map is the $X$-fold product of $q$. We may thus write Lemma~\ref{lem:G-basis iff generalized section} as the following pre-image statement.
\begin{Corollary} \label{cor:Fr_G(F) preimage Sym(X)}
    Let $F$ be a free $G$-set, set $X=F/G$ and denote by $\Sym(X)$ the group of bijections of $X$. The set $\Fr(F)$ inside $F^X$ is the pre-image of $\Sym(X)$ under $q^X$.
\end{Corollary}
We remark that the set $\Sym(X)$ is a special subset of $X^X$; the latter has a natural monoid multiplication given by composition, and $\Sym(X)$ is the set of units of this monoid.

We will now start to take topology into account; we assume $G$ is a topological group and $F$ is a $G$-space. The set $F/G$ thus inherits the quotient topology, such that $q$ is continuous. For any set $X$, $F^X$ inherits the product topology and the map $q^X$ is continuous. In this setting, Cor.~\ref{cor:Fr_G(F) preimage Sym(X)} could be used to deduce topological properties of $\Fr(F)$. This requires us to consider the topology of $X$, which should be homeomorphic to the topology of $F/G$ given the bijectivity from Lemma~\ref{lem:G-basis iff free}. We will restrict ourselves to discrete $X$, i.e.\ $F$ should be a semi-torsor; this preserves the index set intuition and guarantees that previous results on $G$-sets are still valid. Pictorially, this means that the orbits inside $F$ form a partition of $F$ into open and closed subsets. Under this assumption, the following holds.
\begin{Corollary} \label{cor:FrG(F) open and closed in FX}
    If $F$ is a $G$-semi-torsor, then $\Fr(F)$ is open and closed in $F^X$.
\end{Corollary}
\begin{proof}
    As $X$ is discrete, so is $X^X$ (with the product topology). Hence $\Sym(X)$ is an open and closed subset of $X^X$. As $q^X$ is continuous, the claim follows from Cor.~\ref{cor:Fr_G(F) preimage Sym(X)}.
\end{proof}
In case $G$ is a Lie group and $F$ is a $G$-manifold, this corollary implies that $\Fr(F)$ is an open and closed submanifold of $F^X$. We assume $G$ to be finite-dimensional in this paper, but as we allow for infinite $X$, the manifold $F^X$ can be infinite-dimensional.

We remark that the similar version of Cor.~\ref{cor:FrG(F) open and closed in FX} for vector spaces need not hold. For $V$ a real vector space of finite dimension $n$, this statement would be that the set $\Fr(V)$ inside $V^n$ is open and closed. However, $\Fr(V)$ is only an open subset of $V^n$. This can be proven by observing that $\Fr(V)$ inside $V^n$ is similar to $\GL(n,\R)$ inside the space of all real $n\times n$ matrices. Now $\GL(n,\R)$ is the pre-image $\det^{-1}(\R\setminus\{0\})$ and so open. It is not closed; there are sequences of invertible matrices that converge to a non-invertible matrix.
This argument does not hold in the case of $G$-spaces, if we assume that the orbits are disjoint. Given a sequence $\tilde{f}_k$ of bases, look at the induced sequence $q\circ \tilde{f}_k$ of orbit-representatives in $X$. As $X$ is discrete, a sequence in $X$ converges if and only if it becomes constant. That is, if the original sequence $\tilde{f}_k$ converges, then the basis elements settle in particular orbits, which must be different. As the orbits are disjoint, also the limit of the $\tilde{f}_k$ has basis elements in different orbits. That is, the limit of the bases $\tilde{f}_k$ is itself a basis.

Let us go over some examples. First, suppose $G$ is the trivial group and $F$ is a finite set. We will find that the set of bases is the set of tuples listing all elements exactly once. This means that $|\Fr(F)|=|F|!$, and because of this we will sometimes use factorial notation for induced maps on frames later on. We also continue Ex.~\ref{ex:winding torus 1} of the torus winding around multiple times.
\begin{Example}
    Let $F$ be a finite set. We may label its elements such that $F=\{F_1,\dots,F_n\}$, where $n=|F|$, which implicitly means that we chose the index set $X=I_n:=\{1,\dots,n\}$. Any set can be viewed as a group set under the action by the trivial group, and indeed $I_n$ is bijective to the quotient set. The quotient map $q\from F \to I_n$ can be written as the map $F_i\mapsto i$. 
    
    A candidate basis $\tilde{f}\in F^X$ is a map $I_n \to F$. Hence $q\circ \tilde{f}$ maps $i\in I_n$ to the label $j$ defined by $\tilde{f}(i)=F_j$. As we have seen, $\tilde{f}$ is a basis if and only if this map is bijective. This means that the $\tilde{f}(i)$ should be a permutation of the $F_j$, in other words, $\tilde{f}(i)$ is a tuple that list all elements of $F$ exactly once. The set of these tuples is commonly denoted as $F!$, where $|F!|=|F|!$.
\end{Example}

\begin{Example} \label{ex:winding torus 2}
    Let us continue Ex.~\ref{ex:winding torus 1} on the winding torus, defined by the map $\Pi_k \from T^2 \to S^1$. In particular, let us consider the frame space of its model fiber. We found this model fiber to be $\sqcup^k S^1$, but let us view it now as $S^1\times I_k$, i.e.\ we pick an arbitrary ordering of the circles for our convenience. A frame of $S^1\times I_k$ is a tuple $((z_j,i_j))_{j=1}^k$, with $z_j\in S^1$ and $i_j\in I_k$, such that $j\mapsto i_j$ is a permutation of $I_k$. Concerning its topological properties, we see that varying the $z_j$ yields a $k$-dimensional torus $T^k:=\prod^k S^1$. Furthermore, one obtains such a torus for any permutation of $I_k$, i.e.\ the tori can be indexed by $\Sym(I_k)=S_k$. This means that $\Fr(S^1\times I_k)$ is homeomorphic to $T^k \times S_k$, i.e.\ $k!$ many tori of dimension $k$.
\end{Example}

\subsection{Wreath product action}

In preparing the claim that the frame bundles with model fiber of the form $\Fr(F)$ are principal bundles, it is natural to show that $\Fr(F)$ is a torsor. Intuitively speaking, this will describe how one can construct a new basis from a given one. This reasoning is similar to arguing that $\GL(n,\R)$ can turn any basis of a real $n$-dimensional vector space $V$ into another, which is captured by its action on the set of bases $\Fr(V)$. An important property of this action is that the precise vector space $V$ need not be known; the existing basis vectors are just mixed in such a way that a new basis is obtained. This observation is crucial in defining the frame bundle of a vector bundle, and in some sense we will show a $G$-space equivalent here.

Let us start with the action on tuples of $G$-set elements, i.e.\ the action on the set of functions $F^X=\Hom(X,F)$. A first symmetry is that any element of the tuple may be scaled independently from the others. This is given explicitly by the action 
\begin{equation}
    \begin{split}
        G^X \times F^X &\to F^X\\
        (\tilde{g},\tilde{f}) &\mapsto \tilde{g}\cdot \tilde{f},
    \end{split}
\end{equation}
where $\tilde{g}\cdot \tilde{f}$ is the function defined by point-wise multiplication; $(\tilde{g}\cdot \tilde{f})(x)=\tilde{g}(x)\cdot \tilde{f}(x)$. Of course, this is the natural $G^X$-set structure on $F^X$ inherited from $F$. 

Another symmetry of $F^X$ is to change labels, or in the tuple view, to rearrange positions. More algebraically; as $F^X$ is a function space, it has a natural action by $\Sym(X)$ given by pre-composing with the inverse permutation:
\begin{equation}
    \begin{split}
        \Sym(X) \times F^X &\to F^X\\
        (\sigma,\tilde{f}) &\mapsto \tilde{f}\circ \sigma^{-1}.
    \end{split}
\end{equation}

These actions of $G^X$ and $\Sym(X)$ on $F^X$ can be merged into a single action by a specific group, namely the wreath product of $G$ and $X$. There are multiple constructions known under the term wreath product \cite{Suzuki1982GroupI}, but we shall use it as follows. We only need the information that $G$ is a group and $X$ is a set. The key construction is the canonical (left) action of $\Sym(X)$ on $G^X$ given by $\sigma \cdot \tilde{g}=\tilde{g}\circ \sigma^{-1}$, i.e.\ again shuffling the tuple $\tilde{g}$ according to $\sigma$. As the map $\tilde{g} \mapsto \tilde{g}\circ \sigma^{-1}$ is an automorphism of the group $G^X$, the action defines a map $\Sym(X) \to \Aut(G^X)$. That is, the action specifies a semi-direct product $G^X \rtimes \Sym(X)$ with multiplication given as
\begin{equation}
    (\tilde{g},\sigma)\cdot(\tilde{g}',\sigma')=(\tilde{g}(\tilde{g}'\sigma^{-1}),\sigma\sigma').
\end{equation}
This group is known as a wreath product and written as $G \wr_X \Sym(X)$. This notation allows for other groups than $\Sym(X)$ to act on $X$, but as we will only use $\Sym(X)$ we simply denote this group by $G\wr X$. 

Let us briefly review some basic properties on this wreath product. For $G=\mathbb{F}^\times$, i.e.\ the unit group of a field $\mathbb{F}$, and $X=I_n$, the wreath product is faithfully represented by $n\times n$ generalized permutation matrices. These are matrices in $\GL(n,\mathbb{F})$ having only a single non-zero entry in each row and column. An explicit representation is $(\tilde{g},\sigma)\mapsto \mathrm{diag}(g_1,\dots,g_n)P_\sigma$, where $P_\sigma$ is a standard permutation matrix. If $G$ is a topological/Lie group, then $G\wr X$ is as well via the semi-direct product realization. The connected subgroup of $G\wr X$ is that of $G^X$. Hence, in case $G$ is a Lie group with algebra $\g$, the algebra of $G \wr X$ is $\g^X$, which we view as the space of functions $X \to \g$. It follows that $\dim(G \wr X)=\dim(G)|X|$, so for infinite $X$ and non-discrete $G$ one formally obtains an infinite dimensional object. However, because $G \wr X$ locally looks like separate copies of $G$ and we assume $G$ to be finite dimensional we do not dwell on this.

Having introduced the wreath product, we may now define how it acts on $F^X$. We simply combine the previous actions on $F^X$ by stating that a pair $(\tilde{g},\sigma) \in G^X \times \Sym(X)$ acts by first acting with $\sigma$ and then acting by $\tilde{g}$. When thinking of the wreath product in terms of generalized permutation matrices, this action resembles matrix-vector multiplication by viewing $\tilde{f}$ as a column vector.
\begin{Proposition}\label{prop:GwrX action}
    The map
    \begin{equation}
        \begin{split}
        G\wr X \times F^X &\to F^X\\
        ((\tilde{g},\sigma),\tilde{f}) &\mapsto \tilde{g}(\tilde{f}\circ \sigma^{-1}).
        \end{split}
    \end{equation}
    defines a $G\wr X$-action on $F^X$. Moreover, $\Fr(F)$ is an invariant subset, even a $G\wr X$-torsor.
\end{Proposition}
\begin{proof}
    Let us check that the proposed map indeed defines a $G\wr X$-action. Clearly the identity element of $G\wr X$ fixes any $\tilde{f}$, and for the homomorphism property one checks
    \begin{equation}
        (\tilde{g},\sigma) [(\tilde{g}',\sigma') \tilde{f}]=(\tilde{g},\sigma) [\tilde{g}' (\tilde{f} \sigma'^{-1})]=\tilde{g}(\tilde{g}' \sigma^{-1}) (\tilde{f} \sigma'^{-1}\sigma^{-1})=[\tilde{g}(\tilde{g}' \sigma^{-1}),\sigma\sigma')] \tilde{f}.
    \end{equation}
    So indeed $(\tilde{g},\sigma) [(\tilde{g}',\sigma') \tilde{f}]=[(\tilde{g},\sigma)(\tilde{g}',\sigma')] \tilde{f}$, hence the above map defines a $G\wr X$-action.
    
    To check the restriction to $\Fr(F)$, we check the criterion of Lemma~\ref{lem:G-basis iff generalized section}. If $\tilde{f}\in \Fr(F)$, then $q\tilde{f}$ is invertible and hence
    \begin{equation}
        q((\tilde{g},\sigma)\tilde{f})=q(\tilde{g}(\tilde{f}\sigma^{-1}))=q\tilde{f}\sigma^{-1}
    \end{equation}
   is invertible. Hence $(\tilde{g},\sigma)\tilde{f}\in \Fr(F)$, implying that the action restricts to $\Fr(F)$.
  Concerning the last claim, pick two frames $\tilde{f}$ and $\tilde{f}'$. Then $q\tilde{f}$ and $q\tilde{f}'$ are invertible, hence one obtains a permutation $\sigma=(q\tilde{f}')^{-1}q\tilde{f}$. It follows that $\sigma \cdot \tilde{f}$ and $\tilde{f}'$ have the same underlying permutation as $q(\sigma \cdot \tilde{f})=q\tilde{f}\sigma^{-1}=q\tilde{f}(q\tilde{f})^{-1}q\tilde{f}'=q\tilde{f}'$. Hence for each $x\in X$, the elements $(\sigma\cdot \tilde{f})(x)$ and $\tilde{f}'(x)$ lie in the same $G$-orbit, and as $F$ is free they are related by a unique group element $\tilde{g}(x)$. This means that $\sigma\cdot \tilde{f}$ and $\tilde{f}'$ are related by a unique element of $G^X$. Hence any two frames can be related by a unique element of $G\wr X$, i.e.\ the action is free and transitive.
\end{proof}

\begin{Example} \label{ex:winding torus 3}
    We already saw that $\Fr(S^1\times I_k)$ is homeomorphic to $T^k \times S_k$ in Ex.~\ref{ex:winding torus 2}. The action of $U(1)\wr I_k$ on this space is clear; $U(1)^k$ rotates the angles in $T^k$, and $S_k$ permutes both factors of $T^k \times S_k$. Indeed, $T^k \times S_k$ is a $U(1)\wr I_k$-torsor.
\end{Example}

We wish to make a technical remark on the choice of group. Observe that the symmetry group of the basis $\Fr(F)$ can also be argued to be $\Aut(F)$. Indeed, given any two basis $\tilde{f}$ and $\tilde{f}'$, the map $\phi_{\tilde{f}}'\phi_{\tilde{f}}^{-1}$ is an element of $\Aut(F)$, even the unique element which brings $\tilde{f}$ to $\tilde{f}'$. It follows that this action on $\Fr(F)$ is also free and transitive. Hence one could pick either $G \wr X$ or $\Aut(F)$ as the symmetry group of $F$. However, there is a subtle yet significant difference between the two.

Let us illustrate this difference in the more familiar setting of a real vector space $V$ of dimension $n$. Clearly, $\Fr(V)$ has a natural action by $\GL(n,\R)$; a matrix $A=(A_{ij})\in \GL(n,\R)$ acts on a basis as
\begin{equation}
    A\cdot (v_1,\dots,v_n)=(A_{1i}v_i,\dots,A_{ni}v_i),
\end{equation}
where repeated indices are summed over. However, one may argue that the symmetry group of $\Fr(V)$ is $\GL(V)$, as this is the automorphism group of $V$. The action of $\phi \in \GL(V)$ on a basis is
\begin{equation}
    \phi \cdot (v_1,\dots,v_n)=(\phi(v_1),\dots,\phi(v_n)).
\end{equation}
This is again free and transitive, similar to the $\GL(n,\R)$-action. Of course, $\GL(V)$ is isomorphic to $\GL(n,\R)$. However, let us try to use these actions to define new actions, but now on $\Fr(V')$, where $V'$ is another real vector space of dimension $n$. We see that the $\GL(n,\R)$-action can be copied verbatim; $A$ acting on a basis $(v_1',\dots,v_n')$ of $V'$ is simply $(A_{1i}v_i',\dots,A_{ni}v_i')$. This is not so for the $\GL(V)$-action; $\phi \in \GL(V)$ takes elements of $V$ and not of $V'$. Even in the case $V=\R^n$ this subtle issue arises. In this notation, under the $\GL(n,\R)$-action every vector in the new basis depends on all the old ones, whereas under the $\GL(\R^n)$-action every new vector depends only on the previous one. In particular, if one perturbs only one of the basis vectors, then this may change all vectors under the $\GL(n,\R)$-action, but only one under the $\GL(\R^n)$-action. As a final remark, note that vector bundles describe a smooth family of vector spaces. When defining the frame bundle of a vector bundle, it is this crucial that a global $\GL(n,\R)$-action can be defined, despite the fact that every fiber involves another vector space.

For $G$-sets we find a similar argument; instead of using $\Aut(F)$ we will use $G\wr X$. This will be important when we discuss bundles involving $G$-spaces later on. A more detailed exploration of $\Aut(F)$ and how it is isomorphic to the wreath product $G\wr X$ is given in Appendix~\ref{sec:Aut of G-set}.

\subsection{Functorial property}

The previous material indicates that a $G$-semi-torsor $F$ can be mapped to the $G\wr (F/G)$-torsor $\Fr(F)$. Such a statement hints at a functorial property, which brings us to the question which equivariant maps $\alpha \from F_1\to F_2$ will induce a map on frames $\Fr(F_1)\to \Fr(F_2)$ in the canonical way $\tilde{f} \mapsto \alpha \circ \tilde{f}$. As we assume that $F_1$ and $F_2$ are free, this can be answered by looking at orbits only. We first show that $\alpha$ induces a unique map on the quotients, and that invertibility of this map is equivalent to the posed question.
\begin{Lemma} \label{lem:frame map iff alpha/ iso}
    Let $F_i$ be $G_i$-semi-torsors with orbits spaces $X_i=F_i/G_i$ for $i=1,2$. Let $\alpha \from F_1 \to F_2$ be a $\xi$-equivariant map. Then $\alpha$ descends to a unique map on orbits $\alpha/ \from F_1/G_1 \to F_2/G_2$. That is, there is a unique arrow such that the following diagram commutes.
    \begin{equation} \label{eq:diagram alpha/ G-set}
        \begin{tikzcd}
            F_1 \ar{r}{\alpha} \ar{d}{q_1} & F_2 \ar{d}{q_2}\\
            F_1/G_1 \ar[dashed]{r}{\alpha/} & F_2/G_2
        \end{tikzcd}
    \end{equation}
    Moreover, an equivariant map $\alpha \from F_1\to F_2$ lifts to a map $\alpha !\from \Fr(F_1)\to \Fr(F_2)$ if and only if $\alpha/ \from X_1 \to X_2$ is a bijection, i.e.\ if $\alpha$ is a bijection on orbit level. In this case, and after identifying $X_1$ and $X_2$ with $X$ for simplicity, $\alpha!$ is equivariant w.r.t.\ the group map $\xi!=(\xi^X,\id_{\Sym(X)})\from G_1\wr X \to G_2\wr X$.
\end{Lemma}
\begin{proof}
    By equivariance, $\alpha$ maps a $G_1$-orbit in $F_1$ inside a single $G_2$-orbit in $F_2$. That is, the map $q_2 \circ \alpha$ is $G_1$-invariant. Hence $\alpha/$ exists as a function. More generally, as the $X_i$ are discrete, this map is also continuous.
    Let us thus consider the induced maps on bases. The statement is trivial in case $X_1$ and $X_2$ are not bijective; then bases of $F_1$ and $F_2$ have unequal cardinality, but also $\alpha/$ cannot be a bijection. Hence we treat the case $X_1\cong X_2\cong X$. In this case, there is a well-defined function $\alpha^X \from F_1^X \to F_2^X$ given by $\tilde{f}\mapsto \alpha \circ \tilde{f}$. To check that $\alpha \circ \tilde{f}$ is a basis of $F_2$ whenever $\tilde{f}$ is a basis of $F_1$, we use the criterion in Lemma~\ref{lem:G-basis iff generalized section}. This concerns the invertibility of
    \begin{equation}
        q_2\circ \alpha\circ \tilde{f}=\alpha/ \circ q_1 \circ \tilde{f}.
    \end{equation}
    From this equation it follows that for $q_2\circ \alpha\circ \tilde{f}$ to be invertible whenever $q_1\circ \tilde{f}$ is, it is both necessary and sufficient that $\alpha/$ is a bijection.
    Concerning equivariance, note that $\alpha^X$ is $\xi^X$-equivariant. Moreover, $\alpha^X$ is also $\Sym(X)$-equivariant as
    \begin{equation}
        \sigma(\alpha^X(\tilde{f}))=\sigma(\alpha\tilde{f})=\alpha\tilde{f}\sigma^{-1}=\alpha^X(\tilde{f}\sigma^{-1})=\alpha^X(\sigma\tilde{f}).
    \end{equation}
    That is, $\alpha^X$ is equivariant with respect to both $\xi^X$ and $\id_{\Sym(X)}$, which implies that $\alpha^X$ is equivariant w.r.t.\ the map $\xi! \from G_1\wr X \to G_2\wr X$ given by $(\xi^X,\id_{\Sym(X)})$. As the actions restrict to the frames, the restriction $\alpha!$ of $\alpha^X$ to the frames is again equivariant w.r.t.\ $\xi!$.
\end{proof}

The requirement in this result is not automatic; there exist equivariant maps of semi-torsors that are not a bijection on orbit level. A simple example is mapping $G\sqcup G$ to $G$ via identity maps. Hence, taking the frame space can only form a functor from semi-torsors to torsors if the maps $\alpha$ between semi-torsors are restricted to those where $\alpha/$ is a bijection. Let us thus define the following categories.
\begin{Definition}
    Define $\STor$ to be the category whose objects are semi-torsors and whose maps are equivariant morphisms $\alpha$ so that $\alpha/$ is a bijection. Given a Lie group $G$, write $\STor_G$ for the restriction of $\STor$ to $G$-semi-torsors and $\id_G$-equivariant morphisms.
\end{Definition}
By restricting the category of group-spaces to torsors, one obtains the category $\Tor$ of torsors, and also the category $\Tor_G$ by requiring the group to be $G$ and morphisms to be $\id_G$-equivariant. It is clear that $\Tor$ is a full subcategory of $\STor$. Taking the frame space then defines a retract in the following sense.
\begin{Theorem} \label{thm:functor semi-torsor -> torsor}
    The assignment
    \begin{equation}
        \begin{split}
            F &\mapsto \Fr(F)\\
            \alpha &\mapsto \alpha!,
        \end{split}
    \end{equation}
    defines a covariant functor $\STor \to \Tor$. Moreover, this functor is (essentially) constant on $\Tor$.
\end{Theorem}
\begin{proof}
    Observe that the source is a well-defined category; any identity map trivially satisfies the condition of Lemma~\ref{lem:frame map iff alpha/ iso}, and composition of maps preserves this condition; if $\alpha$ and $\beta$ are composable maps which satisfy that $\alpha/$ and $\beta/$ are bijections, then as $(\alpha\beta)/=\alpha/\beta/$ also $\alpha\beta$ does. By Prop.~\ref{prop:GwrX action} and Lemma~\ref{lem:frame map iff alpha/ iso}, the assignment is well-defined on objects resp.\ maps. Given the explicit form of $\alpha!$ in Lemma~\ref{lem:frame map iff alpha/ iso}, it follows that $\id_F!=\id_{\Fr(F)}$ and $(\alpha\beta)!=\alpha!\beta!$. Hence the assignment is a functor. As $\Fr(F)$ is a $G\wr X$-torsor and $\alpha!$ is equivariant, this functor takes image in the category of torsors. If $F$ is already a torsor, then $X$ consists of a single point and the condition of Lemma~\ref{lem:G-basis iff generalized section} is vacuous, so $\Fr(F)=F^X\cong F$. Hence, the functor is (essentially) constant on the torsors.
\end{proof}

If one fixes the structure group $G$, then one can say more. The image of $\STor_G$ consists of torsors whose structure groups are of the form $G\wr X$. Given a particular torsor in this image, up to isomorphism the original semi-torsor can be recovered by taking a quotient. The idea is that a semi-torsor $F$ can be recovered from $\Fr(F)$ by considering all elements that appear in any particular position $x\in X$ in the tuples. This position we may chose ourselves, e.g.\ we only consider the first position. This is equivalent to quotienting $\Fr(F)$ by the subgroup $G\wr (X-x)$ of all group elements that do not touch position $x$. This yields a method applicable to any $G\wr X$-torsor, and we obtain an equivalence of categories. In order to show this, it is convenient to first treat the following lemma. This shows that the model $G\wr X$-torsor, being the space $G\wr X$ viewed as a torsor, indeed reduces to $G\times X$, which we expect as $\Fr(G\times X)$ is isomorphic to $G\wr X$ (as torsors). The action we will use can be obtained by viewing elements of $G\wr X$ as generalized permutation matrices, and $(h,x)\in G\times X$ as a column vector consisting of all zeros, except for $h$ at place $x$.
\begin{Lemma} \label{lem:torsor to semi-torsor}
    There is an action of $G\wr X$ on $G\times X$ defined by
    \begin{equation}
        (\tilde{g},\sigma)\cdot (h,x)=(\tilde{g}(\sigma(x))h,\sigma(x)).
    \end{equation}
    The stabilizer of $(h,x)$ is the subgroup $G\wr (X-x)$.
    Moreover, the induced map $\frac{G\wr X}{G\wr (X-x)} \to G\times X$, where the quotient is endowed with the $G$-action on slot $x$, is an isomorphism in $\STor_G$.
\end{Lemma}
\begin{proof}
    To check the action property, observe that
    \begin{equation}
        \begin{split}
            [(\tilde{g},\sigma)(\tilde{h},\tau)]\cdot (z,x)&=(\tilde{g}(\sigma\tau(x))\tilde{h}(\tau(x))z,\sigma\tau(x))\\
            &=(\tilde{g},\sigma)\cdot (\tilde{h}(\tau(x))z,\tau(x))=(\tilde{g},\sigma)\cdot [(\tilde{h},\tau)\cdot (z,x)].
        \end{split}
    \end{equation}
    Furthermore, $(\tilde{g},\sigma)$ lies in the stabilizer of $(h,x)$ if and only if $\sigma(x)=x$ and $\tilde{g}(x)=1$. That is, the place indexed by $x$ is not touched. Hence the stabilizer is canonically identified with $G\wr (X-x)$. Viewing $G\wr X$ as a $G$-space as indicated, the orbit map $G\wr X \to G\times X$ is $G$-equivariant, and so the quotient yields an isomorphism of $G$-semi-torsors.
\end{proof}
An equivalence of categories can now be obtained as follows. To avoid confusion concerning the class of all discrete topological spaces, we restrict to the finite case.
\begin{Proposition} \label{prop:cat eq stor}
    For a fixed Lie group $G$, the frame functor induces an equivalence of categories
    \begin{equation}
        \STor_G^{<\infty} \to \coprod_{n=1}^\infty \Tor_{G\wr I_n},
    \end{equation}
    where the superscript $<\infty$ indicates the restriction to finite index sets.
\end{Proposition}
\begin{proof}
    Let us show that this restricted functor is both fully faithful and essentially surjective. Given a $G\wr I_n$-torsor $T$, so that $T$ is isomorphic to the space $G\wr I_n$ viewed as torsor, by Lemma~\ref{lem:torsor to semi-torsor} one has $\Fr\left(T/(G\wr (I_n-1))\right) \cong G\wr I_n \cong T$, and hence the functor is essentially surjective.
    To show it is fully faithful, let $F,F'$ be two $G$-semi-torsors with equal index set $I_n$. Let $\bar{\alpha} \from \Fr(F) \to \Fr(F')$ be any map of $G\wr I_n$-torsors. Pick any frame $\tilde{f}\in \Fr(F)$. Using its image frame $\bar{\alpha}(\tilde{f})$, define the map $\alpha:=\phi_{\bar{\alpha}(\tilde{f})}\circ \phi_{\tilde{f}}^{-1} \from F \to F'$, which does not depend on our choice of $\tilde{f}$. Clearly, $\alpha$ is the unique morphism $F\to F'$ satisfying $\Fr(\alpha)=\bar{\alpha}$. Hence the functor is also fully faithful, and the equivalence follows.
\end{proof}

\section{Semi-principal bundles and their frame bundles}
\label{sec:semi-prin and frame bundle}

We will now use the results on the frames of a semi-torsor to introduce the frame bundle of a semi-principal bundle. This association is of functorial nature. We also discuss how a connection on a semi-principal bundle can be defined and how it carries over to frame bundle, which happens in such a way that the parallel transport can be described more easily.

\subsection{The frame bundle of a semi-principal bundle}

In Sec.~\ref{sec:bases of group-sets}, we saw that a semi-torsor has an associated space of frames, which is itself a torsor. This hints that a semi-principal $G$-bundle $\pi \from B \to M$ with model fiber $F$ has an associated frame bundle, which is a principal bundle. This is indeed what we will show next.

We will stay close to the earlier argument, following that $\Fr(F)$ is a special subspace of $F^X$, where again $X$ is an index set bijective to $F/G$. For this analogy, we need the $X$-fold fiber-wise product bundle $\pi^{\odot X} \from B^{\odot X}\to M$ of $B$. It is convenient to realize $B^{\odot X}$ as a pull-back of the $X$-fold direct product bundle
\begin{equation}
    \begin{tikzcd}
        F^X \ar{r} & B^X \ar{r}{\pi^X} & M^X.
    \end{tikzcd}
\end{equation}
As $B$ is a $G$-space, clearly $B^X$ is a $G\wr X$-space by a similar argument as we used for $F^X$ in Prop.~\ref{prop:GwrX action}. Likewise, $M^X$ is a $\Sym(X)$-space, and $\pi^X$ is equivariant for the canonical group map $G\wr X \to \Sym(X)$. Moreover, $\pi^X$ admits $G\wr X$-equivariant local trivializations. Indeed, let $\phi \from U \times F \to B|_U$ be a local $G$-equivariant trivialization of $B$. This induces the map
\begin{equation}
    \begin{split}
        \phi^X \from U^X \times F^X &\to B^X|_{U^X}\\
        (u(x),\tilde{f}(x)) &\mapsto \phi(u(x),\tilde{f}(x)),
    \end{split}
\end{equation}
which is a local trivialization of $B^X$. It is $G\wr X$-equivariant; the $G^X$-equivariance is by construction, and the $\Sym(X)$-equivariance is immediate, cf the verification in the proof of Lemma~\ref{lem:frame map iff alpha/ iso}.

Let us move to $B^{\odot X}$. Obviously, tuples in $B^X$ may contain elements from different fibers of $B$, and we obtain the fiber-wise product by demanding that they come from a single fiber. This can be expressed geometrically using the constant/diagonal map $c\from M \to M^X$ sending $m\in M$ to the constant function $[x\mapsto m]$. This realizes the fiber-wise product bundle $B^{\odot X}$ as a pull-back, whose corresponding diagram is
\begin{equation}
    \begin{tikzcd}
        B^{\odot X} \ar{d}{\pi^{\odot X}} \ar[hook]{r} & B^X \ar{d}{\pi^X}\\
        M \ar[hook]{r}{c} & M^X
    \end{tikzcd}
\end{equation}
so that both bundles have model fiber $F^X$. In this way the topology and manifold structure of $B^{\odot X}$ are clear, and $\pi^{\odot X}$ is smooth. Moreover, as the image of $c$ is $\Sym(X)$-invariant, it follows that $B^{\odot X}$ inherits a $G\wr X$-action from $B^X$ that leaves $\pi^{\odot X}$ invariant. A local trivialization $\phi$ of $B$ induces the local trivialization $\phi^{\odot X}$ of $B^{\odot X}$ given by
\begin{equation}
    \begin{split}
        \phi^{\odot X} \from U \times F^X &\to B^{\odot X}\\
        (u,\tilde{f}(x)) &\mapsto \phi(u,\tilde{f}(x)),
    \end{split}
\end{equation}
which is $G\wr X$-equivariant. We thus obtain the following result.
\begin{Proposition} \label{prop:BodotX is GwrX-space bundle}
    Given a semi-principal $G$-bundle $\pi \from B \to M$ with index set $X$, then the fiber-wise product bundle $F^X \to B^{\odot X} \to M$ is canonically a $G\wr X$-space bundle.
\end{Proposition}

We can now view the frame bundle $\Fr(B)$ of $B$ as the subspace of $B^{\odot X}$ where the elements define bases, as done formally in the following definition. Given the previous result, it is only a minor step to prove the bundle property of $\Fr(B)$, as we do directly afterwards.
\begin{Definition} \label{def:FrG(B)}
    Given a semi-principal $G$-bundle $\pi\from B\to M$, its frame bundle is defined as the bundle
    \begin{equation} \label{eq:definition of FrG(B)}
        \Fr(B)=\set{\tilde{b}\in B^{\odot X}}{\tilde{b} \text{ is a basis of the corresponding fiber of } B}.
    \end{equation}
\end{Definition}
\begin{Theorem} \label{thm:FrG(B) PFB}
    If $\pi\from B\to M$ is a semi-principal $G$-bundle, then its frame bundle $\Fr(B)$ is an open subbundle of $B^{\odot X}$, and forms a principal bundle
    \begin{equation}
        \begin{tikzcd}
            G\wr X \ar{r} & \Fr(B) \ar{r} & M.
        \end{tikzcd}
    \end{equation}
\end{Theorem}
\begin{proof}
    Given Prop.~\ref{prop:BodotX is GwrX-space bundle}, it suffices to verify that $\Fr(F)$ is locally an open subbundle of $B^{\odot X}$. Let $U\subset M$ be a trivializing open of $B$, consider the induced local trivialization $\phi^{\odot X} \from U\times F^X \to B^{\odot X}|_U$. As $\Fr(F)$ is open and closed in $F^X$ by Cor.~\ref{cor:FrG(F) open and closed in FX}, the map $\phi^{\odot X}$ restricts to a $G\wr X$-equivariant isomorphism from $U \times \Fr(F)$ onto its image. This image is $\Fr(B)|_U$ as for each $u\in U$ the map $\phi$ establishes an isomorphism $F \to B_u$ of $G$-spaces. Hence $\Fr(B)$ is an open subbundle of $B^{\odot X}$ with model fiber $\Fr(F)$. As $\Fr(F)$ is a $G\wr X$-torsor by Prop.~\ref{prop:GwrX action}, the bundle is principal.
\end{proof}

We place the following comment concerning the notation $\Fr(B)$. This notation makes explicit reference to the $G$-manifold $B$, but implicitly uses the bundle map $\pi$ as well. As seen in Ex.~\ref{ex:winding torus 4} below, different projection maps may yield different frame bundles, yet the total space of these projection maps is the same $G$-manifold. We thus write $\Fr(B)$ with the understanding that $B$ is short for the entire bundle structure.
\begin{Example}[Frame bundle of winding torus] \label{ex:winding torus 4}
    In Ex.~\ref{ex:winding torus 3} we saw that $\Fr(S^1\times I_k)$ is the $U(1)\wr I_k$-torsor $T^k\times S_k$. The frame bundle of the bundle $\Pi_k$ can be envisioned as these $k!$ tori moving along the circle, where each torus takes the place of its $(1\; 2\; \dots \; k)$ shifted version (assuming the labels follow the order from $\Pi_k$). In other words, the gluing needed to obtain the frame bundle of $\Pi_k$ from $[0,1]\times (T^k \times S_k)$ can be expressed using the group element $(\id_{U(1)^k},(1\; 2\; \dots \; k))\in U(1)\wr I_k$, as we may assume the angles to match. The total space of the frame bundle thus consists of $(k-1)!$ components, where each component is a concatenation of cyclically permuting tori. We observe that the total space of each $\Pi_k$ is the torus $T^2$, yet the obtained frame bundles are different.
\end{Example}

We wish to finish our treatise on $\Fr(B)$ by emphasizing the need to leave $B$ behind if one wishes to use a well-defined permutation action on the bundle. In the case of $\Fr(B)$, one uses the globally defined order of tuples to define the permutation action. One may ask if the local labelling indexed by $X$ also suffices to define a permutation action. As we discuss in Appendix~\ref{sec:S_n action iff trivial}, the answer is negative in general.

\subsection{Functoriality of taking the frame bundle}
\label{sec:functor}

We finished the previous section with the result that any semi-principal bundle has an associated principal frame bundle. Again, this association can be viewed as a functor, similar to the semi-torsor to torsor functor in Thm.~\ref{thm:functor semi-torsor -> torsor}. Clearly, the objects of the source category are the semi-principal bundles. However, the admissible maps will not be all equivariant bundle maps; inspecting a single fiber brings us back to Lemma~\ref{lem:frame map iff alpha/ iso}, which indicates that a map should be a bijection on orbits. This condition readily lifts to the level of bundles, and is a sufficient condition to obtain an induced map on the frame bundles as proven below.
\begin{Lemma} \label{lem:bundle bases to bases}
    Let $\alpha \from B_1 \to B_2$ be an equivariant bundle map of semi-principal bundles over $M$, where the map on groups is $\xi \from G_1 \to G_2$, and $X_1\cong X_2\cong X$. Then $\alpha$ descends to a fiber-wise map on orbits $\alpha/ \from B_1/G_1 \to B_2/G_2$. That is, in the following diagram the lower arrow exists and is a bundle map.
    \begin{equation} \label{eq:diagram alpha/}
        \begin{tikzcd}
            B_1 \ar{r}{\alpha} \ar{d}{Q_1} & B_2 \ar{d}{Q_2}\\
            B_1/G_1 \ar[dashed]{r}{\alpha/} & B_2/G_2
        \end{tikzcd}
    \end{equation}
    Moreover, $\alpha$ induces a map of bundles
    \begin{equation}
        \begin{split}
            \alpha! \from \Fr(B_1) &\to \Fr(B_2)\\
            \tilde{b} &\mapsto \alpha\circ\tilde{b}
        \end{split}
    \end{equation}
    if and only if $\alpha/ \from B_1/G_1 \to B_2/G_2$ is a bundle isomorphism. In this case, $\alpha!$ is equivariant w.r.t.\ the map $\xi!=(\xi^X,\id_{\Sym(X)}) \from G_1\wr X \to G_2 \wr X$, so in particular $\alpha!$ is a map of principal bundles.
\end{Lemma}
\begin{proof}
    As $\alpha$ is equivariant $\alpha/$ exists as a function, which is fiber-preserving as the action is so. To prove it is smooth, one can look locally over a trivializing open or use that $Q_1$ is a surjective submersion. Hence $\alpha/$ is a well-defined bundle map.
    As $\alpha$ is a bundle map, so is its product $\alpha^{\odot X}\from B_1^{\odot X} \to B_2^{\odot X}$. The restriction of $\alpha^{\odot X}$ is also a bundle map, as $\Fr(B_1)$ is an open subbundle of $B_1^{\odot X}$ by Theorem~\ref{thm:FrG(B) PFB}. This restriction takes image in $\Fr(B_2)$ if and only if $\forall m\in M$ the map $\alpha_m \from (B_1)_m \to (B_2)_m$ maps bases to bases. By Lemma~\ref{lem:frame map iff alpha/ iso}, this is if and only if $\alpha/_m\from (B_1/G_1)_m \to (B_2/G_2)_m$ is a bijection for all $m$. However, as $\alpha/$ is a bundle map this is equivalent to $\alpha/$ being an isomorphism. The equivariance of $\alpha$ readily follows; a similar derivation as in Lemma~\ref{lem:frame map iff alpha/ iso} applies. Hence $\alpha^{\odot X}$ restricts to the map $\alpha! \from \Fr(B_1) \to \Fr(B_2)$ if and only if $\alpha/$ is a bundle isomorphism, in which case $\alpha!$ is a map of principal bundles.
\end{proof}

This result tells us that the frame functor is not applicable to the full subcategory of group-manifold bundles obtained by restricting the fibers to semi-torsors but retaining all possible maps of group-manifold bundles. To obtain a genuine frame functor, we must restrict the maps $\alpha$ to those which satisfy that $\alpha/$ is a bundle isomorphism. Similar to the semi-torsor case in Thm.~\ref{thm:functor semi-torsor -> torsor}, this will yield a subcategory; $\id_B /=\id_{B/G}$ is always an isomorphism, and a composition of admissible maps is again admissible as $(\alpha \beta)/=\alpha/ \beta/$. Because of its interest in the frame bundle construction, we wish to list this category in a definition.
\begin{Definition}
    The \emph{category of semi-principal bundles over a base $M$}, denoted $\SPrin(M)$, is defined to have as objects the semi-principal bundles over $M$ and as morphisms all equivariant bundles maps $\alpha$ such that $\alpha/$ is an isomorphism. Given a Lie group $G$, write $\SPrin_G(M)$ for the restriction of $\SPrin(M)$ to semi-principal $G$-bundles and $\id_G$-equivariant maps.
\end{Definition}
Clearly, as we deliberately discard some maps by the orbit condition, $\SPrin(M)$ is a subcategory but not a full subcategory of the group-manifold bundles over $M$. On the other hand, the principal bundles and their morphisms are still contained in the semi-principal bundles. Let us write $\Prin(M)$ for the category of principal bundles, and $\Prin_G(M)$ for the restriction to principal $G$-bundles and $\id_G$-equivariant maps with $G$ a Lie group. We then find that $\SPrin(M)$ can be seen as a specific extension of $\Prin(M)$.
\begin{Lemma}
    The category $\Prin(M)$ is a full subcategory of $\SPrin(M)$. That is, all principal bundles are semi-principal and $\Hom_{\Prin(M)}(P,P')=\Hom_{\SPrin(M)}(P,P')$ whenever $P$ and $P'$ are principal bundles over $M$.
\end{Lemma}
\begin{proof}
    All principal bundles are semi-principal bundles as any torsor is a semi-torsor. Pick a principal $G$-bundle $P$ and a principal $G'$-bundle $P'$ over $M$, and consider an equivariant bundle map $\alpha \from P \to P'$. As $P/G\cong M$ and $P'/G'\cong M$ and $\alpha/$ preserves fibers, one has $\alpha/=\id_M$ up to isomorphism. That is, the condition that $\alpha/$ is an isomorphism is automatically satisfied. As this is the only axiom that distinguishes between morphisms in $\Prin(M)$ and $\SPrin(M)$, it follows that $\Hom_{\Prin(M)}(P,P')=\Hom_{\SPrin(M)}(P,P')$.
\end{proof}

Back to the frame bundles, one may observe that taking the frame bundle maps $\SPrin(M)$ to its full subcategory $\Prin(M)$. This association is functorial, and can be interpreted as a kind of retract. The formal statement is the following theorem, which is the bundle equivalent of Thm.~\ref{thm:functor semi-torsor -> torsor}.
\begin{Theorem}
    The assignment
    \begin{equation}
        \begin{split}
            B &\mapsto \Fr(B)\\
            \alpha &\mapsto \alpha!
        \end{split}
    \end{equation}
    defines a covariant functor $\SPrin(M) \to \Prin(M)$. Moreover, this functor is (essentially) constant on $\Prin(M)$.
\end{Theorem}
\begin{proof}
    By Theorem~\ref{thm:FrG(B) PFB} this functor is well-defined on the objects, and by Lemma~\ref{lem:bundle bases to bases} it is also well-defined concerning morphisms. It is easily seen that identity maps and compositions are preserved, hence the map is a (covariant) functor. Also, if $B$ is a principal bundle, i.e.\ $X$ is a point, then the basis condition is vacuous and $\Fr(B)=B^X\cong B$.
\end{proof}

\begin{Example}
    An remarkable application of this result is the quotient map $Q \from B \to B/G$. Indeed, this is a morphism in $\SPrin(M)$, where we view $B/G$ as an semi-principal bundle with trivial structure group. The corresponding bundle map diagram is
    \begin{equation} \label{eq:B B/G M triangle}
        \begin{tikzcd}[column sep=1em]
            B \ar{rr}{G} \ar[swap]{rd}{F} && B/G \ar{ld}{X}\\
            &M
        \end{tikzcd}
    \end{equation}
    with model fibers indicated. Applying the functor results in the bundle map diagram 
    \begin{equation} \label{eq:Fr_G(B) Fr(B/G) M triangle}
        \begin{tikzcd}[column sep=1em]
            \Fr(B) \ar{rr}{G^X} \ar[swap]{rd}{ G\wr X} && \Fr(B/G) \ar{ld}{\Sym(X)}\\
            &M
        \end{tikzcd}
    \end{equation}
    where each arrow is a principal bundle with the indicated structure group.
\end{Example}

As with the semi-torsors, also the frame bundle functor induces an equivalence of categories if one fixes the structure group $G$.
\begin{Proposition} \label{prop:cat eq sprin}
    For a fixed Lie group $G$ and base manifold $M$, the frame bundle functor induces an equivalence of categories
    \begin{equation}
        \SPrin_G^{<\infty}(M) \to \coprod_{n=1}^\infty \Prin_{G\wr I_n}(M),
    \end{equation}
    where the superscript $<\infty$ indicates the restriction to finite index sets.
\end{Proposition}
\begin{proof}
    This follows Prop.~\ref{prop:cat eq stor} on semi-torsors; quotienting by $G\wr (I_n-1)$ for each $n$ individually yields again an inverse functor, as can be verified using local triviality.
\end{proof}

We finish this exploration of the category $\SPrin(M)$ with the following results on its morphisms.
\begin{Lemma} \label{lem:local form map SPrin(M)}
    Let $\alpha\from B_1 \to B_2$ be a morphism in $\SPrin(M)$, equivariant w.r.t.\ a group map $\xi \from G_1 \to G_2$. Then $\alpha$ can be expressed locally as
    \begin{equation} \label{eq:zeta}
        \zeta=(\id_U,\xi,\id_X) \from U \times G_1 \times X \to U \times G_2 \times X
    \end{equation}
    where $U\subset M$ is a trivializing open of both $B_1$ and $B_2$.
\end{Lemma}
\begin{proof}
    For $i=1,2$, let $\phi_i \from U_i \times G_i\times X_i \to B_i|_{U_i}$ be local trivializations of $B_i$, where $U=U_1\cap U_2$ is non-empty. As $\alpha/$ is an isomorphism, it follows that $X_1 \cong X_2$, hence call them both $X$. Restricting the $\phi_i$ to $U$ and identifying $X_i\cong X$, one has the map $\zeta=\phi_2^{-1} \alpha \phi_1$. This is a bundle map, hence the $\id_U$ follows. It is also equivariant, hence a $G_1$-orbit $\{u\} \times G_1 \times \{x\}$ is mapped inside the $G_2$-orbit  $\{u\} \times G_2 \times \{x'\}$. The induced map $x \mapsto x'$ on orbits is a bijection by assumption, hence by identifying $X_2$ with $X$ differently if needed one may assume $x'=x$ for all $x\in X$, which proves the $\id_X$ part. The middle map thus remains. Here, note that equivariance reads $\zeta(u,g_1,x)=\xi(g_1)\cdot\zeta(u,e_1,x)$, such that $\zeta$ is fixed by the points $\zeta(u,e_1,x)=(u,g_2(u,x),x)$. This defines a smooth function $g_2(u,x)$, which can be translated to be constant $e_2$ by absorbing it in the $G_2$-part of $\phi_2$ (which does not interfere with the choice for $X$ earlier). However, then $\zeta(u,g_1,x)=\xi(g_1)\cdot\zeta(u,e_1,x)=\xi(g_1)\cdot(u,e_2,x)=(u,\xi(g_1),x)$, meaning $\zeta$ equals the proposed map.
\end{proof}

\begin{Corollary}
    The map $\alpha$ is injective/surjective if and only if $\xi$ is. Moreover, in case it is surjective $\alpha$ defines a principal $\ker(\xi)$-bundle.
\end{Corollary}
\begin{proof}
    It follows that $\zeta$ as above is injective/surjective if and only if $\xi$ is. Hence $\alpha$ is injective/surjective if and only if $\xi$ is. Finally, as $\ker(\xi)$ is a closed normal subgroup of $G_1$, if $\xi$ is surjective there is the principal bundle $\ker(\xi) \to G_1 \to G_2$. Hence $\xi$ locally looks like $\pr_V \from V \times \ker(\xi) \to V$, where $V$ is a trivializing open in $G_2$. Hence $\zeta$ can be restricted to the map $\zeta'=(\id_U,\pr_V,\id_X) \from U \times (V \times \ker(\xi)) \times X  \to U \times V \times X$, which is clearly $\ker(\xi)$-invariant. As any point in $B_2$ has a neighborhood of the latter form, it follows that $\alpha$ defines a principal $\ker(\xi)$-bundle whenever it is surjective.
\end{proof}
This latter result is the remarkable observation that maps of semi-principal bundles may themselves be principal bundles.
We observe that the quotient map $Q\from B \to B/G$ in Prop.~\ref{prop:semi-prin = prin + covering} provides an example.

\subsection{Connections on a semi-principal bundle and its frame bundle} \label{sec:connections}

A remarkable property of principal bundles is that they allow for equivariant parallel transport, which can be formulated using special connection 1-forms. We will now inspect how semi-principal bundles can support a similar structure. This will turn out to be almost a verbatim copy of the theory of principal connections. In addition, a connection on a semi-principal bundle will induce a (principal) connection on its frame bundle. 

One way of writing the axioms for a 1-form $\omega$ to be a principal connection is by assuming two relations involving the $G$-action on the manifold. In other words, these axioms do not need to know more than a $G$-manifold structure. This motivates us to use the following definition, for which we use notation as follows. The action map we write as $A \from G\times F \to F$, where acting with $g$ is denoted as $L_g:=A(g,-) \from F \to F$. One can also consider a specific point $f\in F$, and restrict $A$ to $A_f:=A(-,f) \from G \to F$. The differential $a_f=(dA_f)_e \from \g \to T_fF$, where $\g$ again denotes the Lie algebra of $G$, is known as the infinitesimal action at $f$. We thus have the notation to state the following.
\begin{Definition} \label{def:G-connection}
    Let $G$ be a Lie group with algebra $\g$ and let $F$ be a (left) $G$-manifold. A $\g$-valued 1-form $\omega$ on $F$ is called a \emph{$G$-connection 1-form}, or just \emph{$G$-connection}, if
    \begin{itemize}
        \item $\omega_f\circ a_f=\id_{\g}$ for all $f \in F$
        \item $(L_g)^*\omega=\Ad_g(\omega)$ for all $g\in G$.
    \end{itemize}
\end{Definition}
This definition readily applies to semi-principal bundles, and so principal bundles, as their total spaces are by definition $G$-manifolds. However, these axioms can only hold for certain $G$-manifolds. The left-inverse axiom of $\omega_f$ implicitly states that $a_f$ will be a right-inverse, hence $a_f$ must be injective for all $f\in F$. In other words, the action has to be free on a local level. Again, we will not treat the general theory but consider the case of semi-principal bundles. We remark that we choose the name $G$-connection to avoid confusion. As Prop.~\ref{prop:semi-prin = prin + covering} indicates, a $G$-manifold can be a principal bundle for one bundle projection and a semi-principal bundle for another. The axioms of a $G$-connection 1-form do not make reference to such differences, hence we wish to avoid terminology suggesting such reference.

The case of a $G$-connection on a $G$-semi-torsor $F$ turns out to be special. In this case, $F$ admits exactly one $G$-connection. The key observation is that if the infinitesimal action is invertible, then the $G$-connection is fixed.
\begin{Proposition} \label{prop:semi-torsor G-connection}
    Any $G$-semi-torsor $F$ has a natural $G$-connection $\omega$ given by
    \begin{equation}
        \omega_f=a_f^{-1}.
    \end{equation}
    Moreover, this is the only $G$-connection on $F$.
\end{Proposition}
\begin{proof}
    As the action is free, $a_f$ is injective for all $f\in F$. As $F/G$ is discrete, $\dim(F)=\dim(G)$, hence $a_f$ is a linear isomorphism. Hence the condition $\omega_f\circ a_f=\id_\g$ uniquely specifies $\omega_f=a_f^{-1}$.
    To verify that $\omega$ is a $G$-connection, it remains to check the equivariance. This reads
    \begin{equation}
        (L_g^*\omega)_f(X)=a_{gf}^{-1}(dL_g(X))=\Ad_g(a_f^{-1}(X))=\Ad_g(\omega_f(X)),
    \end{equation}
    where we used the relation $dL_g\circ a_f=a_{gf}\circ \Ad_g$.
\end{proof}

This 1-form can be considered as an extension of (a left-action version of) the usual canonical 1-form on $G$, also known as Maurer-Cartan form. A crucial difference is that the above form $\omega$ does not need any identification of (a part of) $F$ with $G$. Moreover, in case $F$ is $G$ viewed as a (left) $G$-torsor one does obtain the usual formula. Indeed, then $A_g=R_g$ and one finds
\begin{equation}
    \omega_g=a_g^{-1}=(dA_g)_e^{-1}=(dR_g)_e^{-1}=(dR_{g^{-1}})_g
\end{equation}
which is indeed the Maurer-Cartan form in the case of a left-action.

Another perspective on this connection involves the map
\begin{equation}
    \begin{split}
        G\times F &\to F\times_q F\\
        (g,f) &\mapsto (gf,f)
    \end{split}
\end{equation}
which by choice of image is always surjective. For a free action, this map is also injective, and an inverse function exists. This is an invertible map; the dimensions of source and target space coincide, and the differential has full rank. Let us write this smooth inverse as
\begin{equation}
    \begin{split}
        F\times_q F &\to G\times F\\
        (f',f) &\mapsto ([f'/f],f)
    \end{split}
\end{equation}
where $[-/-] \from F\times_q F \to G$ is the smooth map sending $(f',f)$ to the unique element $[f'/f]\in G$ such that $f'=[f'/f]f$. That is, the division map allows one to track the group element relating points in $F$. Its derivative will thus yield an infinitesimal generator, which yields again the unique $G$-connection on $F$ in the following way.
\begin{Proposition}
    The canonical $G$-connection on a $G$-semi-torsor $F$ can alternatively be written as
    \begin{equation}
        \omega_{\gamma(0)}(\dot{\gamma}(0))=\dtzero [\gamma(t)/\gamma(0)],
    \end{equation}
    where $\gamma\from (-\epsilon,\epsilon)\to F$ is a curve in $F$.
\end{Proposition}
\begin{proof}
    As $\dim(F)=\dim(G)$, $\gamma(t)$ and $\gamma(0)$ always lie in the same orbit, hence the stated division is well-defined. It suffices to check the $G$-connection axioms, in fact the left-inverse property alone is already sufficient by Prop.~\ref{prop:semi-torsor G-connection}. Taking $f\in F$ and $Y\in \g$, this verification reads
    \begin{equation}
        \omega_f(a_f(Y))=\dtzero [\exp(tY)f/f]=\dtzero \exp(tY)=Y.
    \end{equation}
    The equivariance can also be verified using the identity $[gf'/gf]=g[f'/f]g^{-1}=C_g([f'/f])$;
    \begin{equation}
        (L_g^*\omega)_{\gamma(0)}(\dot{\gamma}(0))=\dtzero [g\gamma(t)/g\gamma(0)]=\dtzero C_g([\gamma(t)/\gamma(0)])=\Ad_g(\omega_{\gamma(0)}(\dot{\gamma}(0))).
    \end{equation}
\end{proof}
The semi-torsor assumption cannot be omitted; the points $\gamma(t)$ and $\gamma(0)$ should be in the same orbit for any path $\gamma$, hence the dimensions of $F$ and $G$ should be equal. Of course, here we still assume the action to be free, otherwise the division would be ill-defined. Hence $F$ can only consist of separated copies of the standard $G$-torsor, meaning $F$ is a $G$-semi-torsor.

Again, this notation reduces to (the left-action version of) the Maurer-Cartan form in case $F$ is the $G$-torsor $G$. This time, the key relation is $[-/g]=R_{g^{-1}}$, so for a curve $\gamma$ such that $(\gamma,\dot{\gamma})(0)=(g,X)$ it holds that
\begin{equation}
    \omega_g(X)=\dtzero [\gamma(t)/g]=\dtzero R_{g^{-1}}(\gamma(t))=dR_{g^{-1}}(X).
\end{equation}
Nevertheless, the division map does not need to employ a right translation, hence is more general. 

Let us move from semi-torsors to semi-principal bundles. A first check is that a $G$-connection on a semi-principal $G$-bundle induces a $G$-equivariant parallel transport. In other words, we wish to show that this result on principal bundles extends to all semi-principal bundles. One can prove this result in various ways. One way is to use the observation that a semi-principal bundle locally looks like the projection $U \times G \times X \to U$, which locally looks like the projection $U \times G \to U$, hence allows for a similar argument as in principal bundle theory. Alternatively, one can use the result on principal bundles and extend it using the decomposition from Prop.~\ref{prop:semi-prin = prin + covering}. Let us employ the latter method in the following proof.
\begin{Proposition}
    Let $\pi \from B \to M$ be a semi-principal $G$-bundle and $\omega$ a $G$-connection on $B$. Given any piece-wise smooth curve $\gamma \from [0,1] \to M$ and an initial point $b\in B$ above $\gamma(0)$, there is a unique horizontal lift $\Gamma\from [0,1]\to B$ through $b$, specified by $\pi\circ \Gamma=\gamma$ and $\Gamma(0)=b$.
    Moreover, the parallel transport maps are $G$-equivariant; the lift of $\gamma$ starting from $gb$ is $g\Gamma$.
\end{Proposition}
\begin{proof}
    By Prop.~\ref{prop:semi-prin = prin + covering}, the semi-principal bundle $\pi \from B \to M$ decomposes as a principal $G$-bundle $B \to B/G$ and a covering $B/G \to M$. Lifting the curve $\gamma$ from $M$ to a curve $\bar{\gamma}$ in $B/G$ with initial point $[b]=Gb$ then follows from covering theory. Now $\omega$ is a $G$-connection, hence a principal connection concerning the principal $G$-bundle $B \to B/G$. As $b$ lies above $Gb$, one can lift $\bar{\gamma}$ to a unique $\Gamma$ in $B$ such that $\Gamma(0)=b$. Hence any curve $\gamma$ in $M$ lifts to a unique curve $\Gamma$ in $B$ given an initial condition and the 1-form $\omega$. If the starting point would be $gb$, then only the lifting from $B/G$ to $B$ changes. However, as $B \to B/G$ is a principal bundle the lift becomes $g\Gamma$, as claimed.
\end{proof}

We already saw that the frame bundle $\Fr(B)$ of a semi-principal bundle $B$ allows one to describe the symmetries of the fiber of $B$ using an $G\wr X$-action. Similarly, we now want to describe holonomy of the semi-principal bundle $B$ using $G\wr X$. In search of this, we must first discuss how a $G$-connection on $B$ carries over to $\Fr(B)$. This procedure is straightforward as $\Fr(B)$ is a submanifold of $B^{\odot X}$. The latter has a canonically induced connection by stating that a curve
\begin{equation}
    \Gamma(t)=(\Gamma_x(t))_{x\in X}
\end{equation}
in $B^{\odot X}$ is horizontal if and only if the curves $\Gamma_x$ in $B$ are horizontal for all $x\in X$. Equivalently, in the 1-form formalism, the bundle $B^{\odot X}$ is naturally endowed with the $G^X$-connection 1-form $\omega^{\odot X} \in \Omega^1(B^{\odot X},\g^X)$ defined by
\begin{equation}
    (\omega^{\odot X})_{\tilde{b}}:=(\omega_{b_x})_{x\in X}=\prod_{x\in X}\omega_{b_x}.
\end{equation}
We note that this formally leads to infinite dimensional groups and algebras in case $X$ is infinite, but as seen in the curve picture it is well-defined. A connection $\omega!$ on $\Fr(B)$ is obtained by restriction of $\omega^{\odot X}$ to $\Fr(B)$, i.e.\ one takes the pull-back of $\omega^{\odot X}$ along the inclusion $i \from \Fr(B) \to B^{\odot X}$. This induced 1-form is indeed a $G\wr X$-connection. A proof of this requires knowledge of the adjoint representation of $G\wr X$, hence we inspect this first.
\begin{Lemma}
    The adjoint representation $\Ad \from G\wr X \to \GL(\g^X)$ is given by the rules
    \begin{equation}
        \begin{split}
            \Ad_{\tilde{g}}&=\prod_{x\in X} \Ad_{\tilde{g}(x)}\\
            \Ad_\sigma &=[\tilde{X} \mapsto \tilde{X}\sigma^{-1}].
        \end{split}
    \end{equation}
\end{Lemma}
\begin{proof}
    As $(\tilde{g},\sigma)=(\tilde{g},\id_X)(\tilde{e},\sigma)$, it suffices to compute the two listed maps.
    The relevant conjugation maps are
    \begin{equation}
        \begin{split}
            C_{\tilde{g}}(\tilde{h},\tau)&=(C_{\tilde{g}}(\tilde{h}),\tau)=((C_{\tilde{g}(x)}(\tilde{h}(x))_{x\in X},\tau)\\
            C_\sigma(\tilde{h},\tau)&=(\tilde{h}\sigma^{-1},C_\sigma(\tau)).
        \end{split}
    \end{equation}
    The first one recalls that the $\Ad$ of $G^X$ is $X$ copies of the $\Ad$ of $G$. The second immediately shows that derivatives follow the same permutation as the tuple.
\end{proof}

\begin{Proposition} \label{prop:omega! principal connection}
    Let $B$ be a semi-principal $G$-bundle with $G$-connection $\omega$. The 1-form
    \begin{equation}
        \omega!:=i^*\omega^{\odot X} \in \Omega^1(\Fr(B),\g^X)
    \end{equation}
    is a $G\wr X$-connection on $\Fr(B)$. In particular, $\omega!$ is a principal connection.
\end{Proposition}
\begin{proof}
    To start, $\omega!$ indeed takes values in the algebra $\g^X$, which is the algebra of $G\wr X$. The condition that $\omega!$ is a left-inverse for the infinitesimal action on $\Fr(B)$ is satisfied as it is so per copy, i.e.\ for each $x\in X$ independently. To check equivariance of $\omega$, let us consider the $G^X$- and $\Sym(X)$-actions separately.
    The $G^X$-action can again be inspected per $x\in X$, and reduces to the equivariance of $\omega$. Concerning the permutations, fix $\sigma \in \Sym(X)$. Then $(L_\sigma)^*\omega!=\prod_{x\in X}\omega_{\sigma^{-1}(x)}=\Ad_\sigma(\omega!)$. Hence $\omega!$ is a $G\wr X$-connection.
\end{proof}

This statement is a hint that the frame bundle functor extends to bundles with connection, which indeed holds according to the upcoming theorem. The intuition is that the connection information on $B$ and $\Fr(B)$ is in essence the same; the latter consists of copies of the first, compare the equivalence in Prop.~\ref{prop:cat eq sprin}. However, the connection $\omega!$ has the convenient property that its holonomy can be written more explicitly, namely using elements of $G\wr X$.

\begin{Theorem}
    Let $\pi_i \from B_i \to M$ be a semi-principal $G_i$-bundle with $G_i$-connection $\omega_i$ for $i=1,2$. Let $\alpha \from B_1 \to B_2$ be a $\xi$-equivariant morphism compatible with the connections, i.e.\ $\alpha^*\omega_2=\xi_*\omega_1$. Then $\alpha!$ is compatible with the frame bundle connections $\omega_1!$ and $\omega_2!$. That is, the frame bundle functor extends to semi-principal bundles with connection.
\end{Theorem}
\begin{proof}
    The compatibility of $\alpha!$ can be verified entry-wise. Writing $i_k\from \Fr(B_k) \to B_k^{\odot X}$, $k=1,2$, for the inclusions, and using $i_2\circ \alpha!=\alpha^{\odot X}\circ i_1$;
    \begin{equation}
        \begin{split}
            (\alpha!)^*\omega_2!&=(\alpha!)^*i_2^* \omega_2^{\odot X}=i_1^*(\alpha^{\odot X})^* \omega_2^{\odot X}=i_1^*(\alpha^*\omega_2)^{\odot X}\\
            &=i_1^*(\xi_*\omega_1)^{\odot X}=i_1^*(\xi^X)_*\omega_1^{\odot X}=(\xi!)_*\omega_1!
        \end{split}
    \end{equation}
    and so the claim on the frame bundle functor follows.
\end{proof}

We end this section with the remark that a morphism of semi-principal bundles can be used to transfer a connection from one bundle to another. We observe that this statement will also place one in the context of the previous theorem.
\begin{Proposition}
    Let $\pi_i \from B_i \to M$ be semi-principal $G_i$-bundles for $i=1,2$ with equal index set $X$. Let $\alpha \from B_1 \to B_2$ be a $\xi$-equivariant morphism. Then any $G_1$-connection $\omega_1$ on $B_1$ induces a unique $G_2$-connection $\omega_2$ on $B_2$ satisfying
    \begin{equation}
        \alpha^*\omega_2=\xi_*\omega_1.
    \end{equation}
\end{Proposition}
\begin{proof}
    As semi-principal bundles locally look like principal bundles, similarly for the maps between them, one can use the same argument as for principal bundles, e.g.\ \cite[Prop.~II.6.1]{Kobayashi1963FoundationsGeometry}.
\end{proof}

A trivial example is the quotient map $Q\from B \to B/G$; any connection on $B$ gets reduced to the lifting of a curve through a covering map. Another example can be phrased on the winding torus as follows.

\begin{Example} \label{ex:winding torus 5}
    Let us consider the semi-principal bundle $\Pi_k \from T^2 \to S^1$ from Ex.~\ref{ex:winding torus 1}. Define the $U(1)$-connection $\omega_1$ on $T^2$ by declaring the horizontal subspaces to be the lines orthogonal to the vertical bundle, where the metric is obtained via the usual embedding in $\R^3$. Observe that $\omega_1$ only needs to be compatible with the group structure; the projection need not be taken into account. Indeed, the usual embedding of $T^2$ in $\R^3$ corresponds to $\Pi_1$, but the obtained $U(1)$-connection is compatible with any $\Pi_k$.
    
    Consider the map $\alpha \from T^2 \to T^2$ which winds the torus twice along the fibers, which is equivariant for the group map $\xi \from U(1) \to U(1)$ taking the square. This is a bundle self-map with respect to $\Pi_k$, for any $k$. The new connection $\omega_2$ will again be a $U(1)$-connection on $T^2$. In fact, $\omega_2=\omega_1$ as the new horizontal subspaces coincide with the old ones (they are doubly covered by $\alpha_*$). In terms of the formula, we observe that both $\alpha^*$ and $\xi_*$ introduce a scaling by 2, hence cancel.
\end{Example}

\section{Discussion} \label{sec:discussion}

In this paper we explored fiber bundles where the model fiber was a group or group-space. We found these to be significantly different. In addition, we found that principal bundles form a special class of group-space bundles, namely those whose fiber is a torsor.
We then shifted our attention to a intermediate class of group-space bundles, namely those whose fiber is free and has discrete quotient, i.e.\ a semi-torsor. To study these semi-torsors, we defined the notion of a basis of a $G$-set $F$, in analogy with the bases of a vector space. The symmetry group of the space $\Fr(F)$ of bases of $F$ is the wreath product $G\wr X$, with $X=F/G$ the orbit space. The map sending a $G$-semi-torsor $F$ to the $G\wr X$-torsor $\Fr(F)$ is of functorial nature.

This theory was lifted to fiber bundles. The bundles of semi-torsors we called semi-principal bundles, and we found that every semi-principal bundle has a frame bundle, which is the principal $G\wr X$-bundle of frames of the fibers. However, this association is a functor only when restricting the maps of semi-principal bundles to the orbit-preserving ones. The category so obtained has the principal bundles as full subcategory, and taking the frame bundle defines a retracting functor from this category to the principal bundles. Moreover, semi-principals bundles support parallel transport just like principal bundles, and this carries over to the frame bundle in a straightforward way.

As stated in the introduction, our main motivation to study semi-torsors and semi-principal bundles comes from the field of adiabatic quantum mechanics. For the exact way in which these structures arise there, we refer to \cite{Pap2021AMechanics}. In short, one obtains a semi-torsor rather than a torsor as eigenrays can be well-separated for an individual Hamiltonian, yet can be connected by variation of the Hamiltonian. More geometrically, locally eigenrays are distinct objects, yet globally they are part of one and the same structure. This topological property allows for interchanges of eigenstates by adiabatic evolution. The frame bundle functor then recasts this theory into the more familiar setting of principal bundles. In addition to this, we observe a resemblance of the frame bundle theory with the model in \cite{Barkeshli2010U1States}. This studies the fractional quantum Hall effect, in particular electrons that are in a bilayer phase. A $U(1)$-symmetry appears for both layers separately, and in addition there is a $\Z_2$-symmetry describing an exchange. The total group that appears is then $(U(1)\times U(1))\rtimes \Z_2$, i.e.\ $U(1)\wr I_2$ in the notation of this paper, and the gauge fields and their transformations follow the properties of $\omega!$ found here.

\section*{Acknowledgements}

The authors wish to thank Professor Dani{\"e}l Boer for fruitful discussions, and two anonymous referees whose careful remarks helped to improve the quality the paper.

\appendix

\section{Automorphism group of a group-set}
\label{sec:Aut of G-set}

In this appendix we show that the wreath product $G\wr X$ and the automorphism group $\Aut(G\times X)$ of the model free $G$-set $G\times X$ are isomorphic. We first show that $\Aut(F)$ fits in a special short exact sequence (SES), cf the semi-direct product structure of the wreath product. Afterwards, we treat an explicit isomorphism, which can be derived using an argument similar to that of dual spaces.

In both parts, we will use the following notation. Observe that any two bases $\tilde{f},\tilde{f}' \in \Fr(F)$ are related by a unique element $[\tilde{f}'/\tilde{f}]\in \Aut(F)$ specified by the equality
\begin{equation} \label{eq:def [tilde f'/ filde f]}
    \tilde{f'}=[\tilde{f}'/\tilde{f}]\tilde{f}
\end{equation}
of maps $X\to F$. This yields $[\tilde{f}'/\tilde{f}]=\phi_{\tilde{f}'}\circ \phi_{\tilde{f}}^{-1}$, which is a composite of $G$-isomorphisms, hence $[\tilde{f}'/\tilde{f}]$ lies indeed in $\Aut(F)$. This notation is compatible with the more general division in the main text; the division here corresponds to viewing $\Fr(F)$ as an $\Aut(F)$-torsor.

\subsection{Short exact sequence using quotient map}

To start, let us examine the group $\Aut(F)$ for $F$ a free $G$-set. A first observation is that an element of $\Aut(F)$ must induce a bijection on the orbit space $F/G$. That is, the following holds.
\begin{Lemma}
    For any map $\psi\in \Aut(F)$, there is a unique map $C_q(\psi)$ such that the diagram
    \begin{equation}
        \begin{tikzcd}
            F \ar{r}{\psi} \ar{d}{q} & F \ar{d}{q}\\
            F/G \ar{r}{C_q(\psi)} & F/G
        \end{tikzcd}
    \end{equation}
    commutes. This association defines a map of groups
    \begin{equation}
        \begin{split}
            C_q \from \Aut(F) &\to \Sym(F/G)\\
            \psi &\mapsto C_q(\psi).
        \end{split}
    \end{equation}
    which admits sections.
\end{Lemma}
\begin{proof}
    The map $C_q(\psi)$ is well-defined as $q\circ \psi$ is $G$-invariant and so constant on the fibers of $q$. If also $\psi'\in \Aut(F)$, then by combining two of the above squares one can deduce $C_q(\psi'\psi) = C_q(\psi')C_q(\psi)$, i.e.\ $C_q$ is a homomorphism. As $C_q(\id_F)=\id_{F/G}$, it follows that all images of $C_q$ are invertible, hence $C_q$ has image in $\Sym(F/G)$.
    To show that $C_q$ admits sections, let us argue using a section $\tilde{f}$ of $q$. Then for each $\sigma \in \Sym(F/G)$ one has a map
    \begin{equation}
        \psi_{\sigma}:=[\tilde{f}\sigma / \tilde{f}] = \phi_{\tilde{f}} (\id_G,\sigma)\phi_{\tilde{f}}^{-1} \from hf_x \mapsto hf_{\sigma(x)}.
    \end{equation}
    This map is basis dependent, but nevertheless invertible and $G$-equivariant, hence an element of $\Aut(F)$. If one only varies the permutation $\sigma$, one obtains a map of groups
    \begin{equation}
        s_{\tilde{f}} \from \Sym(F/G) \to \Aut(F), \quad \sigma \mapsto \psi_{\sigma}.
    \end{equation}
    To check that this is a section of $C_q$, observe that $C_q(\psi)=q\psi\tilde{f}$ as $\tilde{f}$ is a section of $q$. A computation then shows
    \begin{equation}
        C_q(s_{\tilde{f}}(\sigma))=q\psi_{\sigma}\tilde{f}=q[\tilde{f}\sigma / \tilde{f}]\tilde{f}=q\tilde{f}\sigma=\sigma.
    \end{equation}
\end{proof}

As $C_q$ admits sections it must be surjective, hence let us consider its kernel. This is readily seen to be the subgroup
\begin{equation}
    \Aut(q):=\set{\psi\in \Aut(F)}{q\circ\psi=q}
\end{equation}
i.e.\ the group of orbit-preserving automorphisms. In other words, it is the group of deck transformations of $q$, hence the notation $\Aut(q)$. One thus arrives at the following SES, which can be considered a corollary of the previous lemma. The used notation highlights that this sequence is obtained by considering the automorphisms of the objects in ``$q \from F \to F/G$".
\begin{Corollary} \label{cor:SES and decomp of Aut(F)}
    Let $F$ be a free $G$-set. The group $\Aut(F)$ fits in the short exact sequence
    \begin{equation}
        \begin{tikzcd}
            1 \ar{r} & \Aut(q) \ar{r} & \Aut(F) \ar{r}{C_q} & \Sym(F/G) \ar{r} & 1
        \end{tikzcd}
    \end{equation}
    which is right-split.
\end{Corollary}
This result need not hold for a general $G$-set $F$. In that case, $C_q$ is still a homomorphism with kernel $\Aut(q)$, but it need not be surjective, let alone admit sections. That is, only the part
\begin{equation}
    \begin{tikzcd}
        1 \ar{r} & \Aut(q) \ar{r} & \Aut(F) \ar{r}{C_q} & \Sym(F/G)
    \end{tikzcd}
\end{equation}
is exact. This sequence need not extend to the right as the image of $C_q$ need not be normal in $\Sym(F/G)$. Any $\psi \in \Aut(F)$ may only shuffle orbits of the same shape, hence a counterexample is any $F$ where at least two orbits differ but also two orbits are isomorphic. This does not happen in case $F$ is free, in which case the sequence closes to a SES. There are other assumptions one can place on $F$, but we will not study these in this paper.

\subsection{Autmorphism group of the quotient map}

We now wish to rewrite $\Aut(q)$ in a more explicit form. The elements of this group act on the fibers of $q$, which are isomorphic to $G$, hence they can be specified using an element of $G$ for each orbit. To describe this, we use a division operation on $F$, which is not to be confused with the previous division (on $\Fr(F)$). That is, for $(f',f)\in F\times_q F$ we denote by $[f'/f]\in G$ the unique group element satisfying
\begin{equation}
    f'=[f'/f]f.
\end{equation}
Here, we view $F$ as a $G$-semi-torsor by viewing a set as a discrete space, hence also this division is in line with the main text. However, we wish to state the following familiar rules for this division $[-/-] \from F\times_q F\to G$.
\begin{Lemma}
    Assume $F$ is a free $G$-set, let $(f_1,f_2)\in F\times_q F$. The following rules hold
    \begin{enumerate}
        \item \emph{inverse:} $[f_2 / f_1]^{-1}=[f_1/ f_2]$.
        \item \emph{cancellation:} for any $f\in F$ in the same orbit as $f_1$ and $f_2$, one has $[f_2 / f_1]=[f_2 /f] [f / f_1]$.
        \item \emph{scaling:} for any $g_1,g_2\in G$, one has $[g_2f_2/g_1f_1]=g_2[f_2/f_1]g_1^{-1}$. In particular, for $g_1=g_2=g$, one has $[gf_2/gf_1]=C_g([f_2/f_1])$.
        \item \emph{invariance:} if $\psi \in \Aut(F)$, then $[\psi(f_2)/\psi(f_1)]=[f_2/f_1]$.
    \end{enumerate}
\end{Lemma}
Observe that the division is `covariant' with respect to $\Aut(F)$ in its right slot, whereas in its left slot it is `contravariant'. By this we mean that for $\psi_1,\psi_2 \in \Aut(F)$ and $f \in F$ one has
\begin{equation}
    [f/\psi_1\psi_2(f)]= [f/\psi_1(f)]\cdot [\psi_1(f)/\psi_1\psi_2(f)]= [f/\psi_1(f)]\cdot [f/\psi_2(f)]
\end{equation}
such that the order is preserved, whereas for the other slot it is reversed as
\begin{equation}
    [\psi_1\psi_2(f)/f]=[\psi_1\psi_2(f)/\psi_1(f)]\cdot [\psi_1(f)/f]=[\psi_2(f)/f]\cdot [\psi_1(f)/f].
\end{equation}

We can now specify an isomorphism $\Aut(q)\to G^{F/G}$. Let us use again a section $\tilde{f}$ of $q$; in this way each $\tilde{f}(x)$ lies in the orbit $x$, so no shuffling of orbits will happen in what follows. The idea is that each $\psi\in \Aut(q)$ translates the orbit by a group element, which can be found as
\begin{equation}
    \tilde{g}_{\psi,\tilde{f}}(x)=[\tilde{f}(x)/\psi(\tilde{f}(x))]=[\psi(\tilde{f}(x))/\tilde{f}(x)]^{-1}.
\end{equation}
One thus has the map
\begin{equation}
    \begin{split}
        \Aut(q) &\to G^{F/G}\\
        \psi &\mapsto \tilde{g}_{\psi,\tilde{f}}
    \end{split}
\end{equation}
which is a homomorphism as the division on $F$ is covariant in its right slot. As one can reconstruct $\psi$ from $\tilde{g}_{\psi,\tilde{f}}$ given $\tilde{f}$, this map is invertible, hence an isomorphism. We thus find the following result.
\begin{Lemma} \label{lem:Aut(q) cong G product}
    If $F$ is a free $G$-set, then $\Aut(q)\cong G^{F/G}$.
\end{Lemma}

We wish to place some remarks here. First, this isomorphism depends on the chosen section $\tilde{f}$; picking another section $\tilde{f}'$ of $q$, one has $\tilde{f'}(x)=h_x\tilde{f}(x)$ for some elements $h_x\in G$ and finds
\begin{equation}
    \tilde{g}_{\psi,\tilde{f}'}(x)=[h_x\tilde{f}(x)/\psi(h_x\tilde{f}(x))]=h_x[\tilde{f}(x)/\psi(\tilde{f}(x))]h_x^{-1}=C_{h_x}(\tilde{g}_{\psi,\tilde{f}}(x)).
\end{equation}
This means that, if one restricts to bases which are sections, the identification $\Aut(q)\cong G^{F/G}$ is unique up to inner automorphism of $G^{F/G}$. In particular, the identification is canonical if and only if $G$ is commutative. 
A second remark is that instead of $\tilde{g}_{\psi,\tilde{f}}(x)$ being the element mapping $\tilde{f}(x)$ to $\psi(\tilde{f}(x))$, it is actually the inverse of this. An intuitive answer why this is can be found in the pairing argument we discuss below.

Combining Lemma~\ref{lem:Aut(q) cong G product} with the SES in Cor.~\ref{cor:SES and decomp of Aut(F)} shows that the group $\Aut(F)$ is isomorphic to the semi-direct product $G^{F/G} \rtimes \Sym(F/G)$, which is the wreath product $G\wr (F/G)$. We now wish to discuss this more explicitly in the following.

\subsection{Wreath product as model automorphism group of semi-torsor}

Let us specify an explicit isomorphism $G\wr X \to \Aut(G \times X)$. Here we also allow $G$ to be a topological group or a Lie group, and require automorphisms to respect this additional data. To do this, we assume that $X$ as the orbit space is again discrete, such that $G\times X$ is a model $G$-semi-torsor. An explicit map is given by
\begin{equation}
    \begin{split}
        I \from G\wr X &\to \Aut(G \times X)\\
        \tilde{g} &\mapsto [(g,x) \mapsto (g[\tilde{g}(x)]^{-1},x)]\\
        \sigma & \mapsto [(g,x) \mapsto (g,\sigma(x))]
    \end{split}
\end{equation}
such that a pair is mapped to
\begin{equation}
    I(\tilde{g},\sigma)= [(g,x) \mapsto (g[\tilde{g}(\sigma(x))]^{-1},\sigma(x))].
\end{equation}

\begin{Lemma}
    The map $I \from G\wr X \to \Aut(G \times X)$ is an isomorphism.
\end{Lemma}
\begin{proof}
    Let us check that this map is well-defined. First, we note that the returned functions are indeed (left) $G$-equivariant; $G^X$ acts by right-translation and $\Sym(X)$ does not touch the $G$-slot. These maps are also bijective and continuous, or even smooth, hence lie in $\Aut(G\times X)$. A quick check shows that $I$ is a homomorphism, hence it remains to show that $I$ is invertible.
    
    Take $\psi \in \Aut(G\times X)$. First, note that $C_q$ returns the underlying unique permutation $C_q(\psi)=\sigma_\psi \in \Sym(X)$ of $\psi$, so only the $G^X$-element remains to be found. For this purpose, observe that $\psi\circ (\id_G,(\sigma_\psi)^{-1})$ is identity on the $X$ factor. Hence, one may define $\tilde{g}_\psi$ uniquely by the requirement
    \begin{equation}
        \psi((h,\sigma_\psi^{-1}(x)))=(h\tilde{g}_\psi(x)^{-1},x)
    \end{equation}
    or in function form
    \begin{equation}
        \psi\circ (\id_G,\sigma_\psi^{-1})=r_{\tilde{g}_\psi}^{-1}.
    \end{equation}
    Note that both sides contain left-equivariant maps. The map $\psi \mapsto (\tilde{g}_\psi,\sigma_\psi)$ is the inverse of $I$. As we will see now, it is a homomorphism as well. The combination $\psi'\psi$ clearly has $\sigma_{\psi'\psi}=\sigma_{\psi'}\sigma_{\psi}$, and the similar rule for the $G^X$-element follows from
    \begin{equation}
        \begin{aligned}
            \psi'\psi \circ (\id_G,\sigma_\psi^{-1}\sigma_{\psi'}^{-1})
            &= \psi'(\psi \circ (\id_G,\sigma_\psi^{-1})) \circ (\id_G,\sigma_{\psi'}^{-1})\\
            &= \psi' (\id_G,\sigma_{\psi'}^{-1})(\id_G,\sigma_{\psi'})(r_{\tilde{g}_\psi}^{-1}) \circ (\id_G,\sigma_{\psi'}^{-1})\\
            &= r_{\tilde{g}_{\psi'}}^{-1} \circ r_{\tilde{g}_\psi\circ \sigma^{-1}_{\psi'}}^{-1} =r_{\tilde{g}_{\psi'} (\tilde{g}_\psi\circ \sigma^{-1}_{\psi'})}^{-1}
        \end{aligned}
    \end{equation}
    such that $\tilde{g}_{\psi'\psi}$ equals $\tilde{g}_{\psi'} (\tilde{g}_\psi\circ \sigma^{-1}_{\psi'})$. The latter is exactly the $G$-component in the multiplication of the wreath product. Hence $I$ is an isomorphism of groups.
\end{proof}

The form of $I$ can actually be deduced by an argument similar to that of dual spaces. In particular, this argument explains why $\tilde{g}$ enters via an inverse (in fact $\sigma$ also has an inverse w.r.t.\ previous equations). A motivating argument is as follows. Fix some $f\in F$, let $\tilde{f}$ be a basis with basis index set $X$. Then there is a unique pair $(g,x)\in G\times X$ such that
\begin{equation}
    f=g\tilde{f}(x),
\end{equation}
so the pair $(g,x)$ can be seen as the coordinates of $f$. Let us change basis by acting with $\tilde{g}$ on $\tilde{f}$. Then $f$ is given as
\begin{equation}
    f=g[\tilde{g}(x)]^{-1}(\tilde{g}(x)\tilde{f}(x)),
\end{equation}
so in the new basis the point $f$ has coordinates $(g[\tilde{g}(x)]^{-1},x)$. Hence, the inverse on the group element can be found by looking for the map on coordinates induced by the change of basis.

We discuss this in terms of a pairing map. This map is canonical; it takes coordinates $(g,x)$ and a basis $\tilde{f}$ and returns the element that this pair describes, i.e.\ $g \tilde{f}(x)$. More formally, for the $G$-set $F$ this is the map defined as
\begin{equation}
    \begin{split}
        \pair{\;}{\;} \from (G\times X) \times \Fr(F) &\to F\\
        ((g,x),\tilde{f}) &\mapsto \phi_{\tilde{f}}(g,x)=g \tilde{f}(x).
    \end{split}
\end{equation}
The key observation is the following invariance argument.

\begin{Lemma}
    Any $(\tilde{g},\sigma)\in G\wr X$ induces a unique map $\psi_{\tilde{g},\sigma} \in \Aut(G\times X)$ via the invariance condition
    \begin{equation}
        \pair{\psi_{\tilde{g},\sigma}(g,x))}{(\tilde{g},\sigma)\cdot \tilde{f}}=\pair{(g,x)}{\tilde{f}}, \quad \forall (g,x)\in G\times X,\; \forall \tilde{f}\in \Fr(F).
    \end{equation}
    This assignment $(\tilde{g},\sigma) \mapsto \psi_{\tilde{g},\sigma}$ defines a homomorphism $G\wr X \to \Aut(G\times X)$.
\end{Lemma}
\begin{proof}
    We recall that the $\phi_{\tilde{f}}$ are $G$-isomorphisms for any basis $\tilde{f}$. The invariance condition reads
    \begin{equation}
        \phi_{(\tilde{g},\sigma)\cdot \tilde{f}}(\psi_{\tilde{g},\sigma}(g,x))=\phi_{\tilde{f}}(g,x),\quad \text{ or } \quad \psi_{\tilde{g},\sigma}(g,x)=\phi^{-1}_{(\tilde{g},\sigma)\cdot \tilde{f}}\circ \phi_{\tilde{f}}(g,x).
    \end{equation}
    Hence $\psi_{\tilde{g},\sigma}=\phi^{-1}_{(\tilde{g},\sigma)\cdot \tilde{f}}\circ \phi_{\tilde{f}}$, provided the latter is independent of $\tilde{f}$. This can verified by a direct check, but also by the following argument. Let us use the map $[\tilde{f}'/\tilde{f}]$ and observe that it commutes with the $G\wr X$-action, as it lies in $\Aut(F)$. Moreover, $\phi_{[\tilde{f}'/\tilde{f}]\tilde{f}}=[\tilde{f}'/\tilde{f}]\phi_{\tilde{f}}$ cf Lemma~\ref{lem:G-map properties}. Hence, in the basis $\tilde{f}'$ one gets
    \begin{equation}
        \begin{aligned}
            \psi'_{\tilde{g},\sigma}&=\phi^{-1}_{(\tilde{g},\sigma)\cdot \tilde{f}'}\circ \phi_{\tilde{f}'}=\phi^{-1}_{(\tilde{g},\sigma)\cdot [\tilde{f}'/\tilde{f}]\tilde{f}}\circ \phi_{[\tilde{f}'/\tilde{f}]\tilde{f}}=\phi^{-1}_{[\tilde{f}'/\tilde{f}](\tilde{g},\sigma)\cdot \tilde{f}}\circ \phi_{[\tilde{f}'/\tilde{f}]\tilde{f}}\\
            &=\phi^{-1}_{(\tilde{g},\sigma)\cdot \tilde{f}}[\tilde{f}'/\tilde{f}]^{-1}[\tilde{f}'/\tilde{f}]\phi_{\tilde{f}}=\phi^{-1}_{(\tilde{g},\sigma)\cdot \tilde{f}}\phi_{\tilde{f}}=\psi_{\tilde{g},\sigma}.
        \end{aligned}
    \end{equation}
    Hence $\psi_{\tilde{g},\sigma}$ is a well-defined map. As the $\phi$'s are $G$-isomorphisms, so is $\psi_{\tilde{g},\sigma}$, hence it lies in $\Aut(G\times X)$. The homomorphism claim readily follows from the invariance condition and uniqueness of the counteracting map.
\end{proof}
Given the description $\psi_{\tilde{g},\sigma}=\phi^{-1}_{(\tilde{g},\sigma)\cdot \tilde{f}}\circ \phi_{\tilde{f}}$, one can explicitly compute the induced homomorphism. We then quickly recognize the map $I$ as
\begin{equation}
    \psi_{(\tilde{g},\id)}=\phi_{\tilde{g}\cdot \tilde{f}}^{-1}\circ \phi_{\tilde{f}} \colon (g,x)\mapsto g\tilde{f}(x)=g[\tilde{g}(x)]^{-1}(\tilde{g}(x)\tilde{f}(x))\mapsto (g[\tilde{g}(x)]^{-1},x)
\end{equation}
and analogously
\begin{equation}
    \psi_{(\id,\sigma)}=\phi_{\sigma\cdot \tilde{f}}^{-1}\circ \phi_{\tilde{f}} \colon (g,x) \mapsto g\tilde{f}(x)=g\tilde{f}(\sigma^{-1}(\sigma(x)))\mapsto (g,\sigma(x)).
\end{equation}

We remark that a similar argument can be phrased for vector spaces. In particular, one may consider the map 
\begin{equation}
    \begin{split}
        \R^n \times \Fr(\R^n) &\to \R^n\\
        ((c_i),(v_i)) &\mapsto \sum_{i=1}^n c_iv_i.
    \end{split}
\end{equation}
Note that this explicitly pairs coordinates and bases to describe a certain element. In this case $\GL(n,\R)$ acts on $\Fr(\R^n)$, and the induced map $\GL(n,\R) \to \Aut(\R^n)\cong \GL(n,\R)$ is given by the contragredient map $A \mapsto A^{-T}$.

\section{Permutation action on semi-principal bundle\label{sec:S_n action iff trivial}}

In this appendix we show that a semi-principal bundle $B$ with index set $X$ will in general not support a faithful $\Sym(X)$-action. For simplicity, we will consider the reduced bundle $B/G$, or taking another perspective, we assume $G$ is the trivial group. For the ease of the example, we take $X$ to be finite. Hence $F\cong X=I_n:=\{1,2,\dots,n\}$ for some natural number $n$, and $B$ is a regular covering space.

Some motivation to consider this is to check if an faithful $S_n$-action can be defined already on $B$. If that would be the case, one may wonder if the step to $\Fr(B)$ may be avoided. Concerning the individual fibers there is no problem; choose a bijection to $I_n$ and transfer the action of $S_n=\Sym(I_n)$ on $I_n$ via conjugation. However, it is this choice of bijection that might not be possible on a global scale; it is not guaranteed that all the local choices glue together to form a global $S_n$-action.

Let us check if such a bijection is needed. The case $n=1$ is trivial. The case $n=2$ is also not a problem; one simply declares that the permutation $(12)$ exchanges the only 2 objects in each fiber, which certainly forms a global action.
However, in case $n\geq 3$ one does not obtain a well-defined action in general. Informally, this is because now one really needs to know the labels of fiber elements. Indeed, to specify the action of e.g.\ $(12)$ when there are 3 objects, it is necessary to know which of the objects is labelled by 3. Hence, specifying the action of $(12)$ is equivalent to specifying which of the objects is labelled by 3. This observation leads us to the following result.
\begin{Lemma} \label{lem:action to labelling}
    Let $A$ be a set with $n$ elements, where $3\leq n <\infty$. A faithful $S_n$-action on $A$, or in other words an isomorphism $S_n\cong \Sym(A)$, induces a unique equivariant bijection $A\cong I_n$.
\end{Lemma}
\begin{proof}
    Pick $i\in I_n$, consider the subgroup $H_i$ of permutations of $S_n$ that leave $i$ fixed, i.e.\ the stabilizer of $i$ in $S_n$. Then there is at least one element $a_i\in A$ that is fixed by $H_i$; otherwise $H_i$ would contain a permutation affecting all elements of $I_n$, including $i$, which is a contradiction. This element is unique for $n\geq 3$, as $H_i$ then contains permutations moving every element of $I_n$ except $i$. Hence there is a well-defined map $I_n \to A$, $i\mapsto a_i$.
    This map is equivariant; pick $i\in I$ and $\sigma\in S_n$, then for any $\tau \in H_{\sigma(i)}$ one has $\sigma^{-1}\tau\sigma \in H_{\sigma^{-1}(\sigma(i))}=H_i$. Hence $\sigma^{-1}\tau\sigma a_i=a_i$, or $\tau(\sigma a_i)=\sigma a_i$. As $\tau\in H_{\sigma(i)}$ was arbitrary, $\sigma a_i$ is fixed by $H_{\sigma(i)}$, hence $\sigma a_i=a_{\sigma(i)}$, showing equivariance.
    It remains to show that $i\mapsto a_i$ is a bijection. It suffices to show that all $a_i$ are distinct. This follows from the observation that at least $n$ distinct objects are needed to have a faithful $S_n$-action. As one can explicitly find the bijection $A\cong I_n$ in the above way, it is unique.
\end{proof}

We may use this result per fiber in the bundle theory. In essence, the global $S_n$-action is equivalent to specifying a global labelling, which in turn provides a trivialization of the bundle.
\begin{Proposition} \label{prop:global S_n action -> trivial}
    Let $\pi \from B \to M$ be an $I_n$-bundle, where $3\leq n <\infty$.
    There is an faithful fiber-preserving $S_n$-action on $B$ if and only if $B$ is trivial.
\end{Proposition}
\begin{proof}
    The `if' statement is clear. For the `only if', define $\phi \from B \to M \times I_n$ fiber-wise; for $m\in M$, $b\in B_m$ is mapped to the pair $(m,\#(b))$, where $\#(b)$ is the index obtained by applying Lemma~\ref{lem:action to labelling} on the $S_n$-space $B_m$. Hence $\phi$ is a bijection as it is so at each fiber. Using local triviality of $B$ and continuity of the $S_n$-action, it follows that $\# \from B\to I_n$ is locally constant, hence $\phi=(\pi,\#)$ is a homeomorphism. That is, $B$ is trivial.
\end{proof}

Ironically, this means that one may work with an $S_n$-action on $B$ only if no permutation will appear in parallel transport. The reason why a frame does not have this problem is clear; in that case not a locally defined label but instead the globally defined location in the ordered tuple matters.

\newcommand{\doi}[1]{\href{https://doi.org/#1}{doi:~#1}}
\newcommand{\arxiv}[1]{\href{https://arxiv.org/abs/#1}{arXiv:~#1}}

\end{document}